\documentclass[a4paper,11pt]{amsart}

\usepackage[T1]{fontenc}
\usepackage{amssymb}
\usepackage{mathrsfs}
\usepackage{tikz-cd}
\usepackage[pdfusetitle,linktocpage]{hyperref}
\hypersetup{hidelinks}

\newtheorem{thm}{Theorem}[section]
\newtheorem{prop}[thm]{Proposition}
\newtheorem{lem}[thm]{Lemma}

\newtheorem*{mainthm}{Theorem A}

\theoremstyle{remark}
\newtheorem{rem}[thm]{Remark}

\newcommand{\cG}{\mathcal G}
\newcommand{\cH}{\mathcal H}
\newcommand{\cL}{\mathcal L}
\newcommand{\cB}{\mathscr B}
\newcommand{\BX}{B_X}
\newcommand{\BY}{B_Y}
\newcommand{\BXc}{B_{X,\mathrm c}}
\newcommand{\BYc}{B_{Y,\mathrm c}}
\newcommand{\cBX}{\mathscr B_X}
\newcommand{\cBY}{\mathscr B_Y}
\newcommand{\Ind}{\operatorname{Ind}}
\newcommand{\supp}{\operatorname{supp}}
\newcommand{\op}{\mathrm{op}}
\newcommand{\red}{\mathrm{red}}

\begin{document}
\title{An \texorpdfstring{$L^p$}{Lp} Version of Raeburn's Symmetric Imprimitivity Theorem}
\author{Yeong Chyuan Chung}
\address{School of Mathematics, Jilin University, Changchun 130012, Jilin, P. R. China}
\email{chungyc@jlu.edu.cn}
\date{\today}

\begin{abstract}
Let $1\leq p<\infty$. We prove an $L^p$-operator analogue of Raeburn's symmetric imprimitivity theorem for countable discrete groups. Let $A$ be an $L^p$-operator algebra with a contractive two-sided approximate identity, admitting a nondegenerate isometric representation on a $\sigma$-finite $L^p$-space. Given commuting free and proper actions of groups $G$ and $H$ on a second-countable locally compact Hausdorff space $Z$, together with commuting isometric actions on $A$, we show that the two corresponding full crossed products, and likewise the two reduced crossed products, are Morita equivalent as Banach algebras. The proof is organized through the linking groupoid. Its main analytic ingredient is a coefficient-valued spatial matrix dilation; in the reduced case we prove that this dilation preserves regularity after a Borel reindexing of the opposite corner. A pairing-valued approximate identity yields fullness and nondegeneracy.
\end{abstract}

\maketitle


\section{Introduction}

Raeburn's symmetric imprimitivity theorem \cite{Raeburn88} extends the Green--Rieffel symmetric imprimitivity theorem to coefficients. For commuting free and proper group actions on a locally compact space $Z$, together with commuting actions on a $C^*$-algebra $A$, it gives a strong Morita equivalence between crossed products of the two induced $C^*$-algebras. We prove full and reduced $L^p$-operator analogues for countable discrete groups.

The theory of $L^p$-operator crossed products has developed substantially since Phillips introduced full and reduced crossed products by isometric actions \cite{PhillipsCrossed}; see also the survey \cite{GardellaSurvey}. Subsequent work has treated simplicity and traces \cite{HejazianPooya}, Takai duality \cite{WangZhuTakai}, rigidity for crossed products arising from topologically free actions \cite{ChoiGardellaThiel}, and $p$-nuclearity \cite{WangZhuNuclearity}. Delf\'in, Farsi, and Packer developed twisted Banach algebra crossed products, including full and reduced $L^p$-operator versions and an $L^p$ Packer--Raeburn stabilization theorem \cite{DelfinFarsiPacker}. In a direction especially close to the present work, Wang and Zhu proved an $L^p$-version of Green's imprimitivity theorem for the action of $\mathbb Z$ on $\mathbb Z/n\mathbb Z$, with a concrete spatial matrix-algebra description of the crossed product \cite{WangZhuImprimitivity}. These developments point toward an $L^p$-analogue of the structural theory of $C^*$-crossed products, while also exhibiting phenomena absent from the Hilbert-space theory.

The scalar case $A=\mathbb C$ follows from Morita equivalence of full and reduced $L^p$-operator algebras of equivalent \`etale groupoids; see \cite{CDW}. The passage to a noncommutative coefficient algebra is not formal. For $p\neq2$, the Hilbert-module argument underlying the $C^*$-theory is no longer available, and the full $L^p$ norm need not be an unconditional completion norm. The main point is therefore to compare the intrinsic crossed product norms with norms obtained from a coefficient linking construction.

The bundle and linking-groupoid organization is modeled on groupoid crossed product equivalence arguments such as \cite{MuhlyWilliams,SimsWilliams}, where coefficient algebras are treated as section algebras of bundles and equivalence is encoded by a linking object. Paravicini's Banach-groupoid framework \cite{ParaviciniInduction} provides a Banach-algebraic analogue of this viewpoint. Our analytic argument is different: the $L^p$ setting lacks the Hilbert-module positivity available at $p=2$, and we replace it by a coefficient spatial matrix dilation.

Starting from a covariant $L^p$-representation of either diagonal crossed product, we construct a representation of the compactly supported coefficient algebra of the linking groupoid. The main analytic ingredient is a coefficient version of the spatial matrix dilation in \cite[Section~5]{CDW}. The coordinate decomposition is generally only Borel: it determines contractive band projections (spectral projections when $p=2$), whereas coefficient sections are transported along continuous open bisections. Exact recovery of the original representation on the distinguished corner gives one inequality between norms, while compression to the opposite diagonal corner recovers a covariant representation and gives the reverse inequality. This separation of the Borel and continuous operations is essential to the argument.

For the reduced crossed products, a new point enters: a dilation that begins with a regular representation of one diagonal corner must remain regular on the opposite corner. We prove this by expanding the regular representation into its $G$-coordinates and reindexing the resulting Borel bands using the fibers of $Z\to Z/H$, on each of which $H$ acts freely and transitively. Coefficients are again transported only along continuous open slices.

For fullness, we construct a compactly supported approximate identity for each induced algebra whose members are finite sums of products supported in local trivializations. These products can be realized as off-diagonal convolutions, giving a pairing-valued approximate identity as in the scalar argument of \cite{CDW}. The same approximate identity works for both the full and reduced completions. The resulting pairing-valued approximate identities also establish nondegeneracy of the off-diagonal Banach modules.

The main result is the following.

\begin{mainthm} \label{mainthm}
Let $1\leq p<\infty$. Let $G$ and $H$ be countable discrete groups acting freely and properly on a second-countable locally compact Hausdorff space $Z$, on the left and right, respectively, with commuting actions. Let $A$ be an $L^p$-operator algebra with a contractive two-sided approximate identity, admitting a nondegenerate isometric representation on a $\sigma$-finite $L^p$-space. Suppose that $G$ and $H$ act on $A$ by commuting isometric actions $\alpha$ and $\beta$. Then, with the induced algebras and induced actions defined below,
\[
 F^p\bigl(G,\Ind_H^Z(A,\beta),\sigma\bigr)
 \sim_M
 F^p\bigl(H,\Ind_G^Z(A,\alpha),\tau\bigr),
\]
and
\[
 F^p_{\red}\bigl(G,\Ind_H^Z(A,\beta),\sigma\bigr)
 \sim_M
 F^p_{\red}\bigl(H,\Ind_G^Z(A,\alpha),\tau\bigr).
\]
\end{mainthm}

The contractive approximate identity hypothesis is used in the multiplier-extension, local factorization, and fullness arguments below. We do not claim that it is optimal for the statement of the theorem.

\subsection{Conventions and preliminaries}
Throughout, all Banach algebras and Banach spaces are complex. An $L^p$-operator algebra means a Banach algebra admitting an isometric representation as a closed subalgebra of $B(L^p(\mu))$ for some measure space. All $L^p$-spaces used to define the crossed product norms below are assumed $\sigma$-finite.

We use the Banach-correspondence terminology of \cite[Section~2.3]{CDW}, which is equivalent here to Paravicini's notion of Morita equivalence \cite{ParaviciniMorita}. If $M$ and $N$ are linear subspaces of compatible Banach algebras or modules, then $MN$ denotes the linear span of the products $mn$; its norm closure is written $\overline{MN}$. Products of arbitrary subsets of a groupoid always retain their usual setwise meaning. A Banach algebra $B$ is \emph{nondegenerate} if $\overline{BB}=B$.

A \emph{Banach $B$-pair} is a pair $(X,Y)$ consisting of a left Banach $B$-module $X$, a right Banach $B$-module $Y$, and a bounded bilinear pairing
\[
 \langle\cdot,\cdot\rangle_B:X\times Y\longrightarrow B
\]
satisfying
\[
 \langle bx,y\rangle_B=b\langle x,y\rangle_B,
 \qquad
 \langle x,yb\rangle_B=\langle x,y\rangle_Bb,
 \qquad
 \|\langle x,y\rangle_B\|\leq\|x\|\,\|y\|.
\]
Here $\langle X,Y\rangle_B$ denotes the linear span of all pairing values. It is \emph{full} if $\overline{\langle X,Y\rangle_B}=B$ and \emph{nondegenerate} if $\overline{BX}=X$ and $\overline{YB}=Y$. A \emph{Banach $(A,B)$-correspondence} is a Banach $B$-pair $(X,Y)$ such that $X$ is a Banach $B$--$A$ bimodule, $Y$ is a Banach $A$--$B$ bimodule, and
\[
 \langle xa,y\rangle_B=\langle x,ay\rangle_B
 \qquad(x\in X,\ y\in Y,\ a\in A).
\]
For nondegenerate Banach algebras $A$ and $B$, a \emph{Morita equivalence} consists of bimodules ${}_BX_A$ and ${}_AY_B$ such that $(X,Y)$ is a full nondegenerate Banach $(A,B)$-correspondence, $(Y,X)$ is a full nondegenerate Banach $(B,A)$-correspondence for a pairing ${}_A\!\langle\cdot,\cdot\rangle$, and
\[
 \langle x,y\rangle_Bx'=x\,{}_A\!\langle y,x'\rangle,
 \qquad
 y\langle x,y'\rangle_B={}_A\!\langle y,x\rangle y'.
\]

For later reference we also fix the crossed product conventions. Let $G$ be a countable discrete group acting isometrically on an $L^p$-operator algebra $B$ by $\alpha$. We assume here, as in Theorem~A, that $B$ admits a nondegenerate isometric representation on a $\sigma$-finite $L^p$-space. On the finitely supported functions $C_c(G,B)$ we use
\begin{equation}\label{eq:crossed-conv}
 (f_1*f_2)(g)=\sum_{h\in G}f_1(h)\,\alpha_h\bigl(f_2(h^{-1}g)\bigr).
\end{equation}
A nondegenerate contractive covariant representation on an $L^p$-space $E$ is a pair $(\pi,U)$, where $\pi:B\to B(E)$ is nondegenerate and contractive and $U$ is a group representation of $G$ by surjective linear isometries of $E$, satisfying
\[
 U_g\pi(b)U_g^{-1}=\pi(\alpha_g(b)).
\]
Its integrated form is
\[
 (\pi\rtimes U)(f)=\sum_{g\in G}\pi(f(g))U_g.
\]
The full crossed product norm is
\[
 \|f\|_{F^p(G,B,\alpha)}
 :=\sup_{(\pi,U)}\|(\pi\rtimes U)(f)\|,
\]
where the supremum runs over the preceding covariant representations. Because $B$ admits a nondegenerate isometric representation, the associated regular representation is faithful on $C_c(G,B)$; hence the displayed supremum is a norm. The completion is denoted $F^p(G,B,\alpha)$; this agrees with the standard convention in \cite{PhillipsCrossed}. If $\pi_0:B\to B(E_0)$ is nondegenerate and contractive, its regular covariant representation on $\ell^p(G,E_0)$ is
\begin{equation}\label{eq:regular-conventions}
 \begin{aligned}
 (\pi_0^{\mathrm{reg}}(b)\xi)(s)
 &=\pi_0(\alpha_{s^{-1}}(b))\xi(s),\\
 (\lambda_G(t)\xi)(s)&=\xi(t^{-1}s).
 \end{aligned}
\end{equation}
Here $b\in B$, $\xi\in\ell^p(G,E_0)$, and $s,t\in G$.
The reduced norm is
\[
 \|f\|_{F^p_{\red}(G,B,\alpha)}
 =\sup_{\pi_0}\|(\pi_0^{\mathrm{reg}}\rtimes\lambda_G)(f)\|,
\]
where the supremum runs over all nondegenerate contractive representations $\pi_0$ on $\sigma$-finite $L^p$-spaces. The same separating regular representation shows that this supremum is also a norm. The completion of $C_c(G,B)$ in this norm is denoted $F^p_{\red}(G,B,\alpha)$. The same reduced algebra is often denoted $F^p_\lambda(G,B,\alpha)$.
Finally,
\[
 \|f\|_1:=\sum_{g\in G}\|f(g)\|
\]
defines the usual $\ell^1$ crossed product norm. We denote its completion by
\[
 \ell^1(G,B,\alpha).
\]
We write $b\delta_g$ for the finitely supported function taking the value $b$ at $g$ and zero elsewhere. Both the full and reduced norms are bounded above by $\|\cdot\|_1$.

For a Banach algebra $B$ with a contractive approximate identity, $M(B)$ denotes its multiplier algebra, realized as the algebra of double centralizers, and $ZM(B)$ its center. Every nondegenerate contractive representation of $B$ extends uniquely to a unital contractive representation of $M(B)$; see \cite[Theorem~4.1 and Remark~4.2]{GardellaThiel}. If $T$ is a bounded operator, a bounded operator $S$ is called a \emph{generalized inverse} of $T$ when $TST=T$ and $STS=S$. On an $L^p$-space, by a \emph{multiplication band} we mean the range of multiplication by the characteristic function of a measurable set; we often simply call such a range a band.

We use the following groupoid and bundle terminology. An open subset $C$ of a groupoid $\mathcal G$ is an \emph{open bisection} if the restrictions of both source and range are homeomorphisms onto open subsets of $\mathcal G^{(0)}$. A \emph{Borel bisection} is a Borel subset on which both source and range are one-to-one. A subset $U\subseteq\mathcal G^{(0)}$ is \emph{full} if $r(s^{-1}(U))=\mathcal G^{(0)}$. A Banach algebra bundle $q:\mathscr A\to T$ with fiber $A$ is \emph{locally trivial} if every $t\in T$ has an open neighborhood $V$ and a homeomorphism $q^{-1}(V)\cong V\times A$ whose restriction on each fiber is an isometric algebra isomorphism. We write $\Gamma_0(T,\mathscr A)$ for the continuous sections vanishing at infinity, $\Gamma_c(T,\mathscr A)$ for the compactly supported continuous sections, and $\Gamma_b(T,\mathscr A)$ for the bounded continuous sections.

If $q:\mathscr A\to\mathcal G^{(0)}$ is such a bundle, then
\[
 r^*\mathscr A=\{(\gamma,a):\gamma\in\mathcal G,\ a\in\mathscr A_{r(\gamma)}\}
\]
is its pullback along the range map, with projection $(\gamma,a)\mapsto\gamma$. An action of $\mathcal G$ on $\mathscr A$ is a continuous family of isometric algebra isomorphisms
\[
 \Theta_\gamma:\mathscr A_{s(\gamma)}\longrightarrow\mathscr A_{r(\gamma)}
\]
with $\Theta_u=\operatorname{id}$ on unit fibers and $\Theta_{\gamma\eta}=\Theta_\gamma\Theta_\eta$ for composable arrows. 

\section{The induced algebras and the linking coefficient system}

Throughout Sections~2--6, let $p,G,H,Z,A,\alpha$, and $\beta$ satisfy the hypotheses of Theorem~A, and write $e_G$ and $e_H$ for the identity elements of $G$ and $H$. Set
\[
 X=Z/H,\qquad Y=G\backslash Z.
\]
The commuting actions induce left actions
\begin{equation}\label{eq:quotient-actions}
 g\cdot zH=(gz)H,
 \qquad
 h\cdot Gz=G(zh^{-1}).
\end{equation}
Write
\[
 q_H:Z\to X,\qquad q_G:Z\to Y
\]
for the quotient maps.

\begin{lem}[Simultaneous slices]\label{lem:simultaneous-slices}
For every $z\in Z$ there is an open neighborhood $S$ of $z$ such that
\[
 q_H|_S:S\longrightarrow q_H(S),\qquad
 q_G|_S:S\longrightarrow q_G(S)
\]
are homeomorphisms onto open subsets.
\end{lem}

\begin{proof}
We give the argument for the right $H$-action. Choose a relatively compact open neighborhood $U$ of $z$. Properness and discreteness imply that
\[
 F_H=\{h\in H:\overline U h\cap\overline U\neq\varnothing\}
\]
is finite. For $h\in F_H\setminus\{e_H\}$, freeness gives $zh\neq z$. By the Hausdorff property and continuity of the action, there is an open neighborhood $V_h\subseteq U$ of $z$ such that $V_hh\cap V_h=\varnothing$. Then
\[
 S_H=U\cap\bigcap_{h\in F_H\setminus\{e_H\}}V_h
\]
satisfies $S_Hh\cap S_H=\varnothing$ for every $h\neq e_H$: this is true by construction for $h\in F_H$, and for $h\notin F_H$ it follows from $S_H\subseteq U$. The same argument for the left $G$-action gives an open $S_G\ni z$ with $gS_G\cap S_G=\varnothing$ for $g\neq e_G$. Put $S=S_H\cap S_G$.

Both quotient maps are open, and their restrictions to $S$ are injective. Hence they are homeomorphisms onto the open sets $q_H(S)$ and $q_G(S)$.
\end{proof}

Define
\begin{align*}
 \BX=\Ind_H^Z(A,\beta)
  :=\{d\in C_b(Z,A):&\ d(zh)=\beta_{h^{-1}}(d(z)),\\
   &zH\mapsto\|d(z)\|\text{ belongs to }C_0(X)\},
\end{align*}
and
\begin{align*}
 \BY=\Ind_G^Z(A,\alpha)
  :=\{b\in C_b(Z,A):&\ b(gz)=\alpha_g(b(z)),\\
   &Gz\mapsto\|b(z)\|\text{ belongs to }C_0(Y)\}.
\end{align*}
The commuting actions induce isometric actions on these algebras:
\begin{equation}\label{eq:induced-actions}
 \sigma_g(d)(z)=\alpha_g(d(g^{-1}z)),
 \qquad
 \tau_h(b)(z)=\beta_h(b(zh)).
\end{equation}

The transformation groupoids
\[
 \cG=G\ltimes X,\qquad \cH=H\ltimes Y
\]
are equivalent through $Z$; see, for example, \cite[Example~2.33]{WilliamsToolkit}. We use the convention
\[
 s(g,x)=x,\quad r(g,x)=g\cdot x,
 \qquad
 s(h,y)=y,\quad r(h,y)=h\cdot y.
\]
Let
\[
 \cL=\cG\sqcup Z\sqcup Z^{\op}\sqcup\cH
\]
be the standard linking groupoid associated with the $(\cG,\cH)$-equivalence $Z$; see \cite[Section~2]{SimsWilliams} and \cite[Section~1.4]{MuhlyWilliams}. Its unit space is $X\sqcup Y$. The four pieces can be arranged by range (rows) and source (columns):
\[
 \begin{array}{c|cc}
    & X & Y\\ \hline
  X & \cG & Z\\
  Y & Z^{\op} & \cH
 \end{array}.
\]
This table records the direction of the off-diagonal arrows; it is not a multiplication table. For $z\in Z$ and its opposite arrow $z^{\op}\in Z^{\op}$,
\[
 r(z)=q_H(z),\quad s(z)=q_G(z),
 \qquad
 r(z^{\op})=q_G(z),\quad s(z^{\op})=q_H(z).
\]
The diagonal products are those of $\cG$ and $\cH$, the mixed products are induced by the left and right actions on $Z$ and $Z^{\op}$, and inversion interchanges $z$ and $z^{\op}$. Both $X$ and $Y$ are full clopen subsets of $\cL^{(0)}$. The diagonal pieces are \`etale because $G$ and $H$ are discrete, while Lemma~\ref{lem:simultaneous-slices} shows that every point of $Z$, and hence of $Z^{\op}$, has an open-bisection neighborhood in $\cL$. Since a topological groupoid is \`etale exactly when it is covered by open bisections \cite[Example~2.19]{BKM}, $\cL$ is a locally compact Hausdorff \`etale groupoid.

Consider the associated bundles
\[
 \cBX=(Z\times A)/H\longrightarrow X,
 \qquad
 \cBY=G\backslash(Z\times A)\longrightarrow Y.
\]
These are the associated-bundle realizations of the usual induced algebras; compare the induced algebra construction in \cite{Raeburn88} and the bundle approach to groupoid crossed product equivalence in \cite{MuhlyWilliams}. Explicitly,
\[
 (z,a)\cdot h=(zh,\beta_{h^{-1}}(a)),
 \qquad
 g\cdot(z,a)=(gz,\alpha_g(a)).
\]
We denote the corresponding classes by $[z,a]_H$ and $[z,a]_G$, respectively. If $S$ is an $H$-slice and $x\in q_H(S)$, the bijection $q_H|_S:S\to q_H(S)$ gives a unique $z_x\in S$ with $q_H(z_x)=x$. Every element of the fiber $(\cBX)_x$ has a unique representative of the form $[z_x,a]_H$: any other representative has first coordinate $z_xh$ for a unique $h\in H$, by freeness of the right action, and the quotient relation moves it uniquely back to $z_x$. Consequently
\[
 \cBX|_{q_H(S)}\longrightarrow q_H(S)\times A,
 \qquad
 [z,a]_H\longmapsto(q_H(z),a)\quad(z\in S),
\]
is a local trivialization whose restriction to each fiber is an isometric algebra isomorphism. If $S$ and $S'$ are two $H$-slices and $x\in q_H(S)\cap q_H(S')$, let $z\in S$ and $z'\in S'$ be the representatives of $x$. There is a unique $h(x)\in H$ with $z'=zh(x)$, and the corresponding change of fiber coordinates is
\[
 a\longmapsto\beta_{h(x)^{-1}}(a).
\]
Since $H$ is discrete, the map $x\mapsto h(x)$ is locally constant. Thus the transition maps are isometric algebra automorphisms, and $\cBX$ is a locally trivial Banach algebra bundle. The same argument, using $G$-slices and the automorphisms $\alpha_g$, applies to $\cBY$.

There are isometric algebra isomorphisms
\begin{equation}\label{eq:section-identifications}
 \BX\cong\Gamma_0(X,\cBX),
 \qquad
 \BY\cong\Gamma_0(Y,\cBY),
\end{equation}
given by
\[
 d\longmapsto\widehat d,\quad \widehat d(zH)=[z,d(z)]_H,
 \qquad
 b\longmapsto\widehat b,\quad \widehat b(Gz)=[z,b(z)]_G.
\]
The equivariance conditions make these formulas independent of the chosen representative, and the norm conditions are exactly the vanishing-at-infinity conditions. We shall henceforth use \eqref{eq:section-identifications} without changing notation, so that, for example, $d(x)$ denotes the value of the corresponding section in the fiber $(\cBX)_x$.

The diagonal groupoid actions are
\begin{align*}
 \Theta_{(g,zH)}:(\cBX)_{zH}&\longrightarrow(\cBX)_{(gz)H},
 &\Theta_{(g,zH)}([z,a]_H)&=[gz,\alpha_g(a)]_H,\\
 \Theta_{(h,Gz)}:(\cBY)_{Gz}&\longrightarrow(\cBY)_{G(zh^{-1})},
 &\Theta_{(h,Gz)}([z,a]_G)&=[zh^{-1},\beta_h(a)]_G.
\end{align*}
The second target is precisely the fiber over $h\cdot Gz$ from \eqref{eq:quotient-actions}. For $z\in Z$ define
\[
 \Theta_z:(\cBY)_{Gz}\longrightarrow(\cBX)_{zH},
 \qquad
 \Theta_z([z,a]_G)=[z,a]_H,
\]
and let $\Theta_{z^{\op}}=\Theta_z^{-1}$.

\begin{lem}[Linking coefficient action]
The preceding formulas define a continuous action $\Theta$ of $\cL$ on
\[
 \cB=\cBX\sqcup\cBY\longrightarrow X\sqcup Y
\]
by isometric algebra isomorphisms.
\end{lem}

\begin{proof}
The diagonal formulas are well defined because $\alpha$ and $\beta$ commute. For the off-diagonal map, every element of $(\cBY)_{Gz}$ has a unique representative with first coordinate $z$, so this formula is well defined. Under the canonical identifications of both the source and target fiber with $A$, it is the identity map; in particular it is an isometric algebra isomorphism with inverse $\Theta_{z^{\op}}$.

The mixed composition identities reduce to $\alpha_g\beta_h=\beta_h\alpha_g$, while the remaining identities follow from the diagonal actions and inversion. To check continuity on an off-diagonal piece, take a simultaneous slice $S$ from Lemma~\ref{lem:simultaneous-slices}. The trivializations of $\cBY$ over $q_G(S)$ and $\cBX$ over $q_H(S)$ both use the unique representative in $S$, and in these trivializations $\Theta_z$ is the identity on the $A$-coordinate. Hence the action is continuous. Continuity on the diagonal pieces is immediate from the given group actions.
\end{proof}

If $C\subseteq\cL$ is an open bisection and $m\in\Gamma_b(s(C),\cB)$, define
\[
 (\Theta_Cm)(r(\gamma))
 =\Theta_\gamma(m(s(\gamma)))
 \qquad(\gamma\in C).
\]
This is well defined because $r|_C$ is one-to-one. Since $\Theta_{\gamma\eta}=\Theta_\gamma\Theta_\eta$ for composable arrows, one has $\Theta_{C_1C_2}=\Theta_{C_1}\Theta_{C_2}$ whenever the composition is defined.

Put
\[
 \BXc=\Gamma_c(X,\cBX),\qquad
 \BYc=\Gamma_c(Y,\cBY).
\]
The diagonal compactly supported section algebras identify with
\[
 C_c(G,\BXc)\xrightarrow{\ \cong\ }\Gamma_c(\cG,r^*\cBX),
 \qquad
 C_c(H,\BYc)\xrightarrow{\ \cong\ }\Gamma_c(\cH,r^*\cBY).
\]
For the $G$-corner the displayed isomorphism is
\begin{equation}\label{eq:groupoid-crossed-identification}
 f\longmapsto\widetilde f,
 \qquad
 \widetilde f(g,x)=f(g)(gx),
\end{equation}
with inverse $\widetilde f\mapsto f$, where $f(g)(y)=\widetilde f(g,g^{-1}y)$. Compactness of the groupoid support implies that only finitely many $g\in G$ occur, and each coefficient has compact support in $X$. A direct calculation gives the convolution formula \eqref{eq:crossed-conv}. The $H$-corner is identical.

For $u\in\cL^{(0)}$, write $\cL^u=r^{-1}(u)$. On $\Gamma_c(\cL,r^*\cB)$ we use the convolution
\begin{equation}\label{eq:conv}
 (f_1*f_2)(\gamma)
 =\sum_{\eta\in\cL^{r(\gamma)}}
 f_1(\eta)\,\Theta_\eta\bigl(f_2(\eta^{-1}\gamma)\bigr).
\end{equation}
Because $\cL$ is \`etale, each range fiber is discrete; its intersection with the compact support of $f_1$ is therefore finite, so the sum in \eqref{eq:conv} is finite.

\begin{lem}\label{lem:induced-basic}
The algebras $\BX$ and $\BY$ are $L^p$-operator algebras and have contractive two-sided approximate identities contained in $\BXc$ and $\BYc$, respectively.
\end{lem}

\begin{proof}
We give the argument for $\BX$. First construct the approximate identity. Let $F\subseteq\BX$ be finite and let $\varepsilon>0$. Choose a compact set $K\subseteq X$ such that
\[
 \|d(x)\|<\varepsilon\qquad(x\notin K,\ d\in F).
\]
For each $x\in K$, local triviality of $\cBX$ and local compactness of $X$ give a trivializing open neighborhood $U_x$ whose closure is compact and contained in a trivializing neighborhood. Choose finitely many of these, say $U_1,\ldots,U_m$, covering $K$. Choose
\[
 \chi\in C_c\Bigl(\bigcup_jU_j\Bigr),\qquad 0\leq\chi\leq1,
 \qquad \chi|_K=1,
\]
and functions $\psi_j\in C_c(U_j)$ such that $0\leq\psi_j\leq1$ and $\sum_j\psi_j=1$ on $\supp(\chi)$. Put $\varphi_j=\chi\psi_j$. Then
\[
 \varphi_j\in C_c(U_j),\qquad
 \sum_j\varphi_j=1\text{ on }K,
 \qquad
 0\leq\sum_j\varphi_j\leq1\text{ on }X.
\]

Fix a trivialization over $U_j$. For $d\in F$, let $d_j:U_j\to A$ denote its coordinate function. Since $\supp(\varphi_j)$ is compact, each $d_j(\supp\varphi_j)$ is compact in $A$; hence the finite union over $d\in F$ is compact. The contractive two-sided approximate identity of $A$ therefore contains an element $a_j$ with $\|a_j\|\leq1$ such that left and right multiplication by $a_j$ approximate every element of this compact set within $\varepsilon$. Let $a_j^{(j)}$ be the constant section in this trivialization and define, extending each summand by zero,
\[
 u=\sum_{j=1}^m\varphi_j a_j^{(j)}\in\BXc.
\]
Pointwise,
\[
 \|u(x)\|\leq\sum_j\varphi_j(x)\leq1,
\]
so $\|u\|\leq1$. If $x\in K$, then for $d\in F$,
\begin{align*}
 \|u(x)d(x)-d(x)\|
 &\leq\sum_j\varphi_j(x)\|a_j^{(j)}(x)d(x)-d(x)\|<\varepsilon,
\end{align*}
and the analogous right-hand estimate holds. If $x\notin K$, then
\[
 \|u(x)d(x)-d(x)\|\leq2\|d(x)\|<2\varepsilon,
\]
and similarly on the right. Directing these choices by finite subsets and tolerances gives a contractive two-sided approximate identity $(u_\lambda)\subseteq\BXc$.

Now fix a nondegenerate isometric representation
\[
 \rho_A:A\longrightarrow B(L^p(\Omega,\mu))
\]
on a $\sigma$-finite $L^p$-space, and choose a countable dense set $(z_n)_{n\geq1}$ in $Z$. Define
\[
 \rho_X(d)=\bigoplus_{n\geq1}\rho_A(d(z_n))
\]
on $\ell^p(\mathbb N,L^p(\Omega,\mu))$. It is contractive, and continuity of $x\mapsto\|d(x)\|$ together with density of $(q_H(z_n))$ gives
\[
 \|\rho_X(d)\|=\sup_n\|d(z_n)\|=\|d\|.
\]
Thus $\rho_X$ is isometric.

It remains to check nondegeneracy. For each fixed $z_n$, evaluation $\operatorname{ev}_{z_n}:\BX\to A$ is surjective: choose a local trivialization near $q_H(z_n)$ and multiply the constant local section with prescribed value by a compactly supported scalar function equal to one at $q_H(z_n)$. Evaluating $u_\lambda d\to d$ at $z_n$ shows that $(u_\lambda(z_n))$ is a contractive approximate identity for $A$. Since $\rho_A$ is nondegenerate,
\[
 \rho_A(u_\lambda(z_n))\xi\longrightarrow\xi
\]
for every $\xi\in L^p(\Omega,\mu)$. Hence $\rho_X(u_\lambda)\xi\to\xi$ first for finitely supported vectors in the $\ell^p$-sum and then for all vectors by density and contractivity. Thus $\rho_X$ is nondegenerate. The proof for $\BY$ is symmetric.
\end{proof}

The compactly supported section algebras $\BXc$ and $\BYc$ are dense in $\BX$ and $\BY$: if $d\in\BX$, choose $\chi\in C_c(X)$ equal to one on a compact set outside which $\|d(x)\|$ is small, and use $\chi d$. Consequently $C_c(G,\BXc)$ is dense in $C_c(G,\BX)$ for the $\ell^1$ norm: on a fixed finite group support, approximate each coefficient in $\BX$ by one in $\BXc$. Since both the full and reduced crossed product norms are bounded by $\|\cdot\|_1$, the dense subalgebra $C_c(G,\BXc)$ in the $G$-corner identification has the same full crossed product completion as $C_c(G,\BX)$, and likewise the same reduced crossed product completion. The corresponding statement holds for the $H$-corner.

\section{Coefficient matrix dilation}

\subsection{Borel bands and local multipliers}

The proof of the main theorem rests on a coefficient version of the spatial matrix dilation from \cite[Section~5]{CDW}. Let $(\pi,U)$ be a nondegenerate contractive covariant representation of $(G,\BX,\sigma)$ on a $\sigma$-finite $L^p$-space $\mathcal E=L^p(\Omega,\mu)$; as above, $U_g=U(g)$ is a surjective isometry for every $g\in G$.

Pointwise multiplication by scalar functions gives a central multiplier action
\[
 C_0(X)\longrightarrow ZM(\BX),
 \qquad
 \varphi\longmapsto M_\varphi,
 \qquad
 (M_\varphi d)(x)=\varphi(x)d(x).
\]
Since $\BX$ has a contractive approximate identity, $\pi$ has a unique contractive extension $\overline\pi:M(\BX)\to B(\mathcal E)$. Define
\begin{equation}\label{eq:rho}
 \rho(\varphi)=\overline\pi(M_\varphi)
 \qquad(\varphi\in C_0(X)).
\end{equation}
The representation $\rho$ is nondegenerate because $\overline{C_0(X)\BX}=\BX$ and $\pi$ is nondegenerate.

For $p\neq2$, \cite[Theorem~2.13]{BKM} shows that the central representation $\rho$ already acts by multiplication operators on the given $L^p(\Omega,\mu)$, since the measure is $\sigma$-finite and hence localizable. For $p=2$, we instead keep the original Hilbert space $\mathcal E$: the contractive representation $\rho:C_0(X)\to B(\mathcal E)$ is a nondegenerate $*$-representation and therefore has a projection-valued spectral measure on the Borel subsets of $X$. This avoids replacing $\mathcal E$ by the possibly non-$\sigma$-finite localizable measure space in \cite[Lemma~2.17]{BKM}.

For completeness, we describe the bounded-Borel extension when $p\neq2$. Write $\rho(\varphi)=M_{\rho_0(\varphi)}$ with a contractive homomorphism $\rho_0:C_0(X)\to L^\infty(\Omega,\mu)$. The nondegenerate multiplier extension is again given by multiplication operators: if $(e_\lambda)$ is a contractive approximate identity for $C_0(X)$ and $f\in C_b(X)$, then $\rho(e_\lambda f)$ converges strongly to the multiplier extension at $f$, and the multiplication operators are strongly closed. We obtain a unital contractive homomorphism
\[
 \overline\rho_0:C_b(X)\longrightarrow L^\infty(\Omega,\mu).
\]
As both algebras are commutative $C^*$-algebras, this homomorphism is a $*$-homomorphism. Restriction to $C_0(X)$, followed by the same multiplication action on $L^2(\Omega,\mu)$, gives a nondegenerate $*$-representation. Nondegeneracy follows because a positive-measure set on which all the multiplication operators vanish would, by semifiniteness, support a nonzero $L^p$-vector annihilated by $\rho(C_0(X))$. The Borel functional calculus of this Hilbert-space representation consists of multiplication operators, which also act contractively on the original $L^p$-space.
The strong countable additivity of the resulting
Borel projections also holds on the original
$L^p$-space. Indeed, for pairwise disjoint
Borel sets $T_n$, the corresponding spectral
projections on $L^2(\Omega,\mu)$ are multiplication
by pairwise disjoint characteristic functions,
whose sum agrees almost everywhere with the
multiplier associated with $\bigcup_n T_n$.
Dominated convergence therefore gives
\[
 P_{\bigcup_n T_n}\xi
 =\lim_{N\to\infty}\sum_{n=1}^N P_{T_n}\xi
 \qquad(\xi\in L^p(\Omega,\mu)),
\]
with convergence in the $L^p$-norm.

In either case let
\[
 \widetilde\rho:\mathcal B_b(X)\longrightarrow B(\mathcal E)
\]
be the resulting bounded-Borel extension, where $\mathcal B_b(X)$ denotes the bounded Borel functions. For every Borel set $T\subseteq X$ define
\[
 P_T=\widetilde\rho(1_T).
\]
When $p\neq2$, $P_T$ is a multiplication-band projection; when $p=2$, it is the projection supplied by the spectral measure on the original Hilbert space. In both cases, the projections are contractive, satisfy $P_SP_T=P_{S\cap T}$, and obey countable additivity in the strong-operator topology. In particular, for pairwise disjoint Borel sets $T_n$ and $\xi\in\mathcal E$,
\begin{equation}\label{eq:band-pythagoras}
 \biggl\|P_{\bigcup_n T_n}\xi\biggr\|_p^p
  =\sum_n\|P_{T_n}\xi\|_p^p.
\end{equation}
For $p\neq2$ this is disjoint-support additivity; for $p=2$ it is orthogonality of spectral subspaces. Thus every $P_T\mathcal E$ and every countable $\ell^p$-sum of such spaces is (up to isometry) an admissible $L^p$-space over a $\sigma$-finite measure space: for $p\neq2$ these are literal $L^p$ bands, and for $p=2$ they are Hilbert spaces, which admit $\sigma$-finite $L^2$ realizations.

The covariance relation for $(\pi,U)$ extends to the Borel functional calculus. For $p=2$, this follows from uniqueness of the spectral measure and covariance of the underlying $*$-representation. For $p\neq2$, the bounded-Borel functional calculus is obtained by uniformly bounded strong-operator limits of continuous functions, and conjugation by $U_g$ preserves these limits. Consequently,
\begin{equation}\label{eq:band-cov}
 U_gP_TU_g^{-1}=P_{gT}.
\end{equation}

In what follows we call the ranges of the $P_T$ \emph{Borel bands}, including at $p=2$, where they are more precisely orthogonal spectral subspaces. Only the projection identities and \eqref{eq:band-pythagoras} will be needed.
We need one local multiplier observation. If $O\subseteq X$ is open, write
\[
 B_{X,O}=\Gamma_0(O,\cBX|_O),
\]
viewed as the closed ideal of $\BX$ obtained by extension by zero. Applying the construction of the compactly supported contractive two-sided approximate identity from the preceding section to the restricted locally trivial bundle $\cBX|_O\to O$ shows that $B_{X,O}$ also has such an approximate identity.

\begin{lem}\label{lem:local-multiplier}
For every open $O\subseteq X$,
\[
 \overline{\pi(B_{X,O})\mathcal E}=P_O\mathcal E.
\]
The induced nondegenerate representation of $B_{X,O}$ on $P_O\mathcal E$ extends uniquely to a unital contractive representation
\[
 \widetilde\pi_O:M(B_{X,O})\longrightarrow B(P_O\mathcal E).
\]
If $m\in\Gamma_b(O,\cBX)$, pointwise left and right multiplication by $m$ define a multiplier of $B_{X,O}$, and the corresponding operator $\widetilde\pi_O(m)$ satisfies
\[
 \|\widetilde\pi_O(m)\|\leq\|m\|_\infty
\]
and commutes with $P_T$ for every Borel $T\subseteq O$.
\end{lem}

\begin{proof}
Since $B_{X,O}=\overline{C_0(O)\BX}$ and $\rho(\varphi)\pi(d)=\pi(\varphi d)$,
\[
 \overline{\pi(B_{X,O})\mathcal E}
 =\overline{\rho(C_0(O))\pi(\BX)\mathcal E}.
\]
Nondegeneracy of $\pi$ gives
\[
 \overline{\rho(C_0(O))\pi(\BX)\mathcal E}
 =\overline{\rho(C_0(O))\mathcal E}.
\]
If $(\varphi_\lambda)$ is a contractive approximate identity for $C_0(O)$, then the scalar operators $\rho(\varphi_\lambda)$ converge strongly to $P_O$. Hence
\[
 \overline{\rho(C_0(O))\mathcal E}=P_O\mathcal E.
\]

For $m\in\Gamma_b(O,\cBX)$, define on $B_{X,O}$
\[
 L_m(d)(x)=m(x)d(x),\qquad R_m(d)(x)=d(x)m(x).
\]
These maps have norm at most $\|m\|_\infty$ and satisfy the double-centralizer identities, so they determine a multiplier of $B_{X,O}$. The representation of $B_{X,O}$ on $P_O\mathcal E$ is nondegenerate by the first part and therefore extends to this multiplier, giving $\widetilde\pi_O(m)$. Pointwise scalar multiplication is central, so $\widetilde\pi_O(m)$ commutes with $\rho(\varphi)$ for $\varphi\in C_0(O)$. The commutant of $\widetilde\pi_O(m)$ is strongly closed; since every $P_T$ belongs to the bounded-Borel strong closure of these scalar operators, $\widetilde\pi_O(m)$ commutes with $P_T$ for every Borel $T\subseteq O$.
\end{proof}

Throughout, we also regard $\widetilde\pi_O(m)$ as an operator on $\mathcal E$ by extending it by zero on $(I-P_O)\mathcal E$. The same convention applies to the local multiplier representations in Section~6. The local multiplier extensions also agree on Borel sets where the underlying continuous coefficient sections coincide. We will use this observation both in the covariance calculation and in the regularity argument.

\begin{lem}\label{lem:borel-localization}
Let $U,V\subseteq X$ be open, let $m\in\Gamma_b(U,\cBX)$ and $n\in\Gamma_b(V,\cBX)$, and let $S\subseteq U\cap V$ be Borel. If $m(x)=n(x)$ for $x\in S$, then
\[
 P_S\widetilde\pi_{U}(m)
 =P_S\widetilde\pi_{V}(n).
\]
\end{lem}

\begin{proof}
If $W\subseteq U$ is open, uniqueness of the multiplier extension for the nondegenerate representation of $B_{X,W}$ shows that the restriction of $\widetilde\pi_{U}(m)$ to $P_W\mathcal E$ is $\widetilde\pi_{W}(m|_W)$. We may therefore restrict both operators to $U\cap V$. Let $r=m-n$. For $\varepsilon>0$ the open set
\[
 W_\varepsilon=\{x\in U\cap V:\|r(x)\|<\varepsilon\}
\]
contains $S$. Since the local multiplier operator commutes with all Borel band projections,
\[
 P_S\widetilde\pi_{U\cap V}(r)
 =P_SP_{W_\varepsilon}\widetilde\pi_{U\cap V}(r).
\]
Since $\widetilde\pi_{U\cap V}(r)$ commutes with $P_{W_\varepsilon}$, and its restriction to $P_{W_\varepsilon}\mathcal E$ agrees with $\widetilde\pi_{W_\varepsilon}(r|_{W_\varepsilon})$, Lemma~\ref{lem:local-multiplier}
gives
\[
\begin{aligned}
\bigl\|P_S\widetilde\pi_{U\cap V}(r)\bigr\|
&\leq
\bigl\|\widetilde\pi_{W_\varepsilon}
(r|_{W_\varepsilon})\bigr\|\\
&\leq \|r|_{W_\varepsilon}\|_\infty
\leq\varepsilon.
\end{aligned}
\]
As this holds for every $\varepsilon>0$, we obtain $P_S\widetilde\pi_{U\cap V}(r)=0$, which is the required equality.
\end{proof}

\subsection{Coordinates and the matrix dilation}

We now choose the coordinate system used in the matrix dilation. The hypotheses of \cite[Lemma~5.5]{CDW} apply to $(\cL,X)$: the linking groupoid is \`etale, $X$ is full and clopen, and $Y=\cL^{(0)}\setminus X$ is a paracompact locally compact Hausdorff space. Concretely, fullness gives for each $y\in Y$ an arrow from $X$ to $y$; an open-bisection neighborhood of that arrow can be cut down so that its source lies in $X$ and its range in $Y$, after which local compactness provides a relatively compact range neighborhood. A locally finite refinement yields relatively compact open sets $V_i\subseteq Y$ and relatively compact open bisections $B_i\subseteq\cL$ such that
\[
 r(B_i)=V_i,\qquad O_i:=s(B_i)\subseteq X.
\]
Because $Y$ is second countable, the locally finite family $(V_i)$ is countable: each nonempty $V_i$ contains an element of a fixed countable basis, and local finiteness permits only finitely many $V_i$ to contain any one basis element. We therefore index it by $i\geq1$.

Set
\[
 R_i=V_i\setminus\bigcup_{j<i}V_j,
 \qquad
 \widetilde B_i=B_i\cap r^{-1}(R_i),
 \qquad
 T_i=s(\widetilde B_i).
\]
Then $(R_i)_{i\geq1}$ is a pairwise disjoint Borel partition of $Y$. For notational uniformity in the matrix construction, adjoin a zeroth coordinate corresponding to the original $X$-corner by setting
\[
 R_0=T_0=X,\qquad B_0=\widetilde B_0=X.
\]
Thus $R_i\subseteq Y$ for $i\geq1$, whereas the exceptional zeroth piece $R_0$ is $X$; in particular
\[
 \cL^{(0)}=R_0\sqcup\bigsqcup_{i\geq1}R_i=X\sqcup Y.
\]
For every $i\geq0$, the Borel bisection $\widetilde B_i$ gives a Borel bijection
\[
 \theta_i:T_i\longrightarrow R_i,
 \qquad
 \theta_i(u)=r\bigl((s|_{\widetilde B_i})^{-1}(u)\bigr),
\]
with $\theta_0=\operatorname{id}_X$. The sets $R_i,T_i$ and $\widetilde B_i$ need not be open. They determine the Borel projections $P_T$ and the coordinate maps $\theta_i$; coefficient sections are moved between fibers only by the continuous bundle action along the open bisections $B_i$. Figure~\ref{fig:coordinate-system} summarizes the geometry of these coordinates.

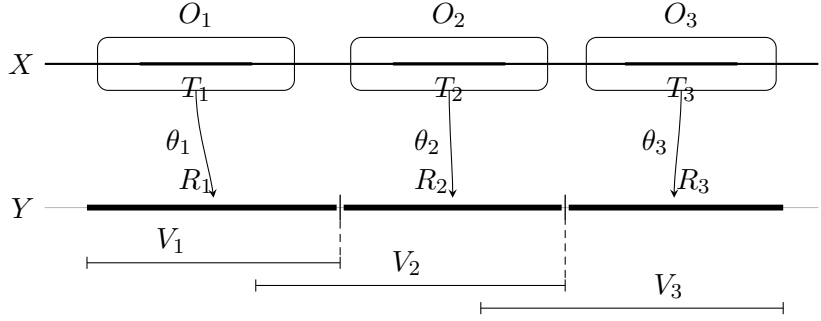
\begin{figure}[htbp]
\centering
\begin{tikzpicture}[x=0.93cm,y=1cm,>=stealth]
  \draw[thick] (0,3.45) -- (11,3.45);
  \node[left] at (0,3.45) {$X$};

  \draw[rounded corners] (0.75,3.10) rectangle (3.55,3.80);
  \node[above] at (2.15,3.80) {$O_1$};
  \draw[very thick] (1.35,3.45) -- (2.95,3.45);
  \node[below] at (2.15,3.41) {$T_1$};

  \draw[rounded corners] (4.35,3.10) rectangle (7.15,3.80);
  \node[above] at (5.75,3.80) {$O_2$};
  \draw[very thick] (4.95,3.45) -- (6.55,3.45);
  \node[below] at (5.75,3.41) {$T_2$};

  \draw[rounded corners] (7.70,3.10) rectangle (10.40,3.80);
  \node[above] at (9.05,3.80) {$O_3$};
  \draw[very thick] (8.25,3.45) -- (9.85,3.45);
  \node[below] at (9.05,3.41) {$T_3$};

  \draw[gray!55,line width=0.4pt]
    (0,1.55) -- (11,1.55);
  \node[left] at (0,1.55) {$Y$};

  \draw[black,line width=2.2pt]
    (0.60,1.55) -- (4.15,1.55);
  \node[above] at (2.15,1.61) {$R_1$};

  \draw[black,line width=2.2pt]
    (4.25,1.55) -- (7.35,1.55);
  \node[above] at (5.5,1.61) {$R_2$};

  \draw[black,line width=2.2pt]
    (7.45,1.55) -- (10.50,1.55);
  \node[above] at (9.225,1.61) {$R_3$};

  \draw[black,line width=0.5pt] (4.20,1.38) -- (4.20,1.72);
  \draw[black,line width=0.5pt] (7.40,1.38) -- (7.40,1.72);

  \draw[->] (2.15,3.10) .. controls (2.15,2.60) and (2.35,2.05) .. (2.40,1.68)
    node[pos=0.45,left] {$\theta_1$};
  \draw[->] (5.75,3.10) .. controls (5.75,2.60) and (5.80,2.05) .. (5.80,1.68)
    node[pos=0.45,left] {$\theta_2$};
  \draw[->] (9.05,3.10) .. controls (9.05,2.60) and (8.95,2.05) .. (8.95,1.68)
    node[pos=0.45,left] {$\theta_3$};

  \draw (0.60,0.82) -- (4.20,0.82);
  \draw (0.60,0.74) -- (0.60,0.90);
  \draw (4.20,0.74) -- (4.20,0.90);
  \node[above left] at (2.15,0.82) {$V_1$};

  \draw (3.00,0.52) -- (7.40,0.52);
  \draw (3.00,0.44) -- (3.00,0.60);
  \draw (7.40,0.44) -- (7.40,0.60);
  \node[above left] at (5.50,0.52) {$V_2$};

  \draw (6.20,0.22) -- (10.50,0.22);
  \draw (6.20,0.14) -- (6.20,0.30);
  \draw (10.50,0.14) -- (10.50,0.30);
  \node[above left] at (9.225,0.22) {$V_3$};

  \draw[densely dashed] (4.20,0.92) -- (4.20,1.38);
  \draw[densely dashed] (7.40,0.62) -- (7.40,1.38);
\end{tikzpicture}
\caption{Coordinate construction, illustrated with three charts. The three lower $V_i$-tracks share a schematic $Y$-coordinate and overlap; the heavy segments on the $Y$-line indicate their disjointification $R_i=V_i\setminus\bigcup_{j<i}V_j$. Each $\theta_i:T_i\to R_i$ is a Borel bijection induced by the open bisection $B_i$ with $T_i\subseteq O_i\subseteq X$. The positions of the $T_i$-segments are purely schematic and do not assert that the Borel sets are intervals, or that the $X$-charts have common coordinates. The zeroth piece $R_0=T_0=X$ is omitted.}
\label{fig:coordinate-system}
\end{figure}

Set
\[
 \mathcal E_i=P_{T_i}\mathcal E,
 \qquad
 \widehat{\mathcal E}=\bigoplus_{i\geq0}^{p}\mathcal E_i.
\]
This countable $\ell^p$-sum is again a $\sigma$-finite $L^p$-space. For every Borel set $W\subseteq\cL^{(0)}=X\sqcup Y$, the set $\theta_i^{-1}(W\cap R_i)$ is automatically contained in $T_i$, so define
\[
 Q_W|_{\mathcal E_i}
 =P_{\theta_i^{-1}(W\cap R_i)}|_{\mathcal E_i}.
\]
For $W=X$, the zeroth block is the identity and every block $i>0$ vanishes because $R_i\subseteq Y$; hence $Q_X$ is the projection onto $\mathcal E_0$. For $W=Y$, the zeroth block vanishes and the $i>0$ blocks are the identities on $\mathcal E_i$. Therefore
\[
 Q_Y=I-Q_X.
\]

For $i>0$, if $d\in\BX$ and $b\in\BY$, define
\[
 m_i(b)=\Theta_{B_i}^{-1}(b|_{V_i})\in\Gamma_b(O_i,\cBX),
\]
and
\[
 \widehat\pi_i(d\oplus b)
 =\widetilde\pi_{O_i}(m_i(b))|_{\mathcal E_i},
 \qquad
 \widehat\pi_0(d\oplus b)=\pi(d).
\]
Let
\[
 \widehat\pi:\BX\oplus\BY\longrightarrow B(\widehat{\mathcal E})
\]
be their direct sum.

\begin{lem}\label{lem:pihat}
The map $\widehat\pi$ is a nondegenerate contractive representation of $\BX\oplus\BY$ on $\widehat{\mathcal E}$.
\end{lem}

\begin{proof}
Multiplicativity and contractivity follow from Lemma~\ref{lem:local-multiplier} and the fact that the bundle maps $\Theta_\gamma$ are isometric algebra isomorphisms.

Fix $i>0$. Since
\[
 \overline{\pi(B_{X,O_i})\mathcal E}=P_{O_i}\mathcal E
\]
and the operators $\widetilde\pi_{O_i}(m)$ commute with $P_{T_i}$, vectors in $\mathcal E_i$ are approximated by finite sums of vectors $\pi(d)\xi$ with $d\in\Gamma_c(O_i,\cBX)$ and $\xi\in\mathcal E_i$. The transported section $\Theta_{B_i}(d)$ has compact support in $V_i$ and therefore extends by zero to an element $b_d\in\BYc$. By construction,
\[
 \widehat\pi_i(0\oplus b_d)\xi=\pi(d)\xi.
\]
Thus $\widehat\pi_i(\BY)\mathcal E_i$ is dense in $\mathcal E_i$. The zeroth block is nondegenerate because $\pi$ is.
\end{proof}

For the matrix construction we also need operators associated with Borel sub-bisections of the transformation groupoid $G\ltimes X$. Let $\mathscr R\subseteq G\ltimes X$ be a Borel bisection contained in a relatively compact open bisection, and put
\[
 S_g=\{x\in X:(g,x)\in\mathscr R\}.
\]
Since $G$ is discrete,
\[
 \mathscr R=\bigsqcup_{g\in F}\{g\}\times S_g,
 \qquad
 F=\{g\in G:S_g\neq\varnothing\},
\]
and every $S_g$ is Borel. Moreover, $F$ is finite because the projection to the discrete group $G$ sends the compact closure of a relatively compact open bisection containing $\mathscr R$ to a compact, hence finite, subset of $G$. Define
\[
 W_{\mathscr R}=\sum_{g\in F}U_gP_{S_g}.
\]

\begin{lem}\label{lem:borel-bisection-operator}
The operator $W_{\mathscr R}$ is a contraction which restricts to a surjective isometry
\[
 P_{s(\mathscr R)}\mathcal E\longrightarrow P_{r(\mathscr R)}\mathcal E,
\]
vanishes on $(I-P_{s(\mathscr R)})\mathcal E$, and has contractive generalized inverse $W_{\mathscr R^{-1}}$. If $\mathscr R$ and $\mathscr S$ are two such Borel bisections, then $W_{\mathscr R} W_{\mathscr S}=W_{\mathscr R\mathscr S}$.
\end{lem}

\begin{proof}
Because $\mathscr R$ is a bisection, the source sets $S_g$ are pairwise disjoint and so are the range sets $gS_g$. Thus the vectors $U_gP_{S_g}\xi$ lie in the ranges of disjoint Borel projections, and \eqref{eq:band-pythagoras} applies. Consequently
\[
 \|W_{\mathscr R}\xi\|_p^p
 =\sum_{g\in F}\|U_gP_{S_g}\xi\|_p^p
 =\sum_{g\in F}\|P_{S_g}\xi\|_p^p
 =\|P_{s(\mathscr R)}\xi\|_p^p.
\]
The second equality uses that each $U_g$ is an isometry, and the last uses that the $S_g$ partition $s(\mathscr R)$. Applying the same calculation to $\mathscr R^{-1}$ gives the surjectivity statement and shows that $W_{\mathscr R^{-1}}$ is a contractive generalized inverse.

If
\[
 \mathscr S=\bigsqcup_{h\in F'}\{h\}\times T_h,
\]
then \eqref{eq:band-cov} gives
\[
 U_gP_{S_g}U_hP_{T_h}
 =U_{gh}P_{T_h\cap h^{-1}S_g}.
\]
These are exactly the source pieces of $\mathscr R\mathscr S$, and their disjointness yields $W_{\mathscr R} W_{\mathscr S}=W_{\mathscr R\mathscr S}$.
\end{proof}

Let $C\subseteq\cL$ be a relatively compact open bisection. For $i,j\geq0$, set
\[
 D_{ij}(C)
 =(B_i^{-1}CB_j)\cap r^{-1}(T_i)\cap s^{-1}(T_j).
\]
This cutdown is equivalently
\[
 D_{ij}(C)=\widetilde B_i^{-1}C\widetilde B_j.
\]
Indeed, if $\gamma=b_i^{-1}cb_j\in B_i^{-1}CB_j$, then
\[
 r(\gamma)=s(b_i),\qquad s(\gamma)=s(b_j).
\]
Since $s|_{B_i}$ and $s|_{B_j}$ are injective,
\[
 r(\gamma)\in T_i\iff b_i\in\widetilde B_i,
 \qquad
 s(\gamma)\in T_j\iff b_j\in\widetilde B_j.
\]
Thus the two descriptions coincide.

For any open or Borel bisection $A$, write $h_A=r\circ(s|_A)^{-1}$ for the induced partial bijection from $s(A)$ onto $r(A)$. It is a homeomorphism when $A$ is open, but only a Borel bijection is asserted for the cutdown $D_{ij}(C)$. The geometry of the $(i,j)$-entry is summarized by the commutative diagram
\[
\begin{tikzcd}[column sep=large,row sep=large]
 s(D_{ij}(C))
   \arrow[r,"h_{D_{ij}(C)}"]
   \arrow[d,"\theta_j"']
 & r(D_{ij}(C))
   \arrow[d,"\theta_i"]
 \\
 s(C)\cap R_j\cap h_C^{-1}(R_i)
   \arrow[r,"h_C"']
 & r(C)\cap R_i\cap h_C(s(C)\cap R_j).
\end{tikzcd}
\]
Indeed, writing an arrow of $D_{ij}(C)$ as $\widetilde b_i^{-1}c\widetilde b_j$ shows directly that its source and range are pulled back through $\theta_j$ and $\theta_i$, respectively.

Only finitely many $D_{ij}(C)$ are nonempty. For $i,j>0$, nonemptiness implies that $r(C)$ meets $R_i\subseteq V_i$ and $s(C)$ meets $R_j\subseteq V_j$; the compact closures of $r(C)$ and $s(C)$ meet only finitely many members of the locally finite family $(V_i)_{i\geq1}$. If $j=0<i$, nonemptiness of $D_{i0}(C)$ forces $r(C)$ to meet $R_i\subseteq V_i$, while if $i=0<j$, nonemptiness of $D_{0j}(C)$ forces $s(C)$ to meet $R_j\subseteq V_j$. Local finiteness therefore gives only finitely many nonzero entries in the exceptional row and column as well; $D_{00}(C)$ contributes at most one further entry.

By Lemma~\ref{lem:borel-bisection-operator},
\[
 W_{ij}(C)=W_{D_{ij}(C)}:\mathcal E_j\longrightarrow\mathcal E_i
\]
is a contraction with generalized inverse $W_{D_{ij}(C)^{-1}}$. Define
\[
 \widehat V_C=(W_{ij}(C))_{i,j}\in B(\widehat{\mathcal E}).
\]

\begin{lem}[Coefficient spatial matrix dilation]
For every relatively compact open bisection $C\subseteq\cL$, the operator $\widehat V_C$ is a contraction which restricts to a surjective isometry
\[
 Q_{s(C)}\widehat{\mathcal E}\longrightarrow Q_{r(C)}\widehat{\mathcal E}
\]
and has generalized inverse $\widehat V_{C^{-1}}$. Moreover,
\begin{equation}\label{eq:Vmult}
 \widehat V_C\widehat V_D=\widehat V_{CD}
\end{equation}
for relatively compact open bisections $C,D$, and
\begin{equation}\label{eq:covariance}
 \widehat V_C\widehat\pi(a)
 =\widehat\pi(\Theta_C(a))\widehat V_C
\end{equation}
for every $a\in\Gamma_0(s(C),\cB)$. If $C'\subseteq C$ is an open sub-bisection, then
\begin{equation}\label{eq:restriction}
 \widehat V_{C'}=Q_{r(C')}\widehat V_C=\widehat V_CQ_{s(C')}.
\end{equation}
\end{lem}

\begin{proof}
For fixed $j$, the initial bands of the nonzero $W_{ij}(C)$ are pairwise disjoint as $i$ varies: otherwise two arrows of the bisection $C$ would have the same source but ranges in two distinct sets $R_i$. Similarly, for fixed $i$, the final bands are pairwise disjoint as $j$ varies. Hence Lemma~\ref{lem:borel-bisection-operator} shows that the finite matrix $\widehat V_C$ is a contraction and is isometric on the direct sum of its initial bands.

For fixed $j$,
\[
 \bigsqcup_i s(D_{ij}(C))
 =\theta_j^{-1}(s(C)\cap R_j).
\]
Indeed, let $u$ belong to the set on the right. There is a unique $\widetilde b_j\in\widetilde B_j$ with $s(\widetilde b_j)=u$. Since $r(\widetilde b_j)=\theta_j(u)\in s(C)$ and $C$ is a bisection, there is a unique $c\in C$ with source $\theta_j(u)$. Its range belongs to a unique $R_i$, and then there is a unique $\widetilde b_i\in\widetilde B_i$ with $r(\widetilde b_i)=r(c)$. Hence $\widetilde b_i^{-1}c\widetilde b_j\in D_{ij}(C)$ has source $u$. The converse is immediate from the definition, and uniqueness of $i$ proves disjointness. Therefore the corresponding band projections sum to
\[
 P_{\theta_j^{-1}(s(C)\cap R_j)},
\]
which is exactly the $j$th block of $Q_{s(C)}$. The symmetric argument gives the blocks of $Q_{r(C)}$. Thus the initial and final projections of $\widehat V_C$ are $Q_{s(C)}$ and $Q_{r(C)}$, respectively, and its generalized inverse is $\widehat V_{C^{-1}}$.

For multiplication one has the disjoint decomposition
\[
 D_{ik}(CD)=\bigsqcup_jD_{ij}(C)D_{jk}(D).
\]
To see the nontrivial inclusion, take $\eta\in D_{ik}(CD)$ and write
\[
 \eta=\widetilde b_i^{-1}cd\widetilde b_k
\]
with composable $c\in C$ and $d\in D$. The common unit $x=s(c)=r(d)$ lies in a unique $R_j$. Let $\widetilde b_j\in\widetilde B_j$ be the unique arrow with $r(\widetilde b_j)=x$. Inserting the unit $x=\widetilde b_j\widetilde b_j^{-1}$ gives
\[
 \eta=(\widetilde b_i^{-1}c\widetilde b_j)
       (\widetilde b_j^{-1}d\widetilde b_k)
 \in D_{ij}(C)D_{jk}(D).
\]
The reverse inclusion is immediate, and uniqueness of $R_j$ proves that the union is disjoint. Lemma~\ref{lem:borel-bisection-operator} then gives \eqref{eq:Vmult} block by block.

For covariance, take $\gamma\in D_{ij}(C)$ and write
\[
 \gamma=b_i^{-1}cb_j,
 \qquad
 b_i\in\widetilde B_i,\ c\in C,\ b_j\in\widetilde B_j.
\]
Since $\Theta$ is a groupoid action,
\[
 \Theta_\gamma=\Theta_{b_i}^{-1}\Theta_c\Theta_{b_j}.
\]
Fix a group slice
\[
 D_{ij}(C)\cap(\{g\}\times X)=\{g\}\times S_g.
\]
If $m\in\Gamma_b(O,\cBX)$, set
\[
 m^g(x)=\Theta_{(g,g^{-1}x)}\bigl(m(g^{-1}x)\bigr)
 \qquad(x\in gO).
\]

By \eqref{eq:band-cov}, $U_g$ maps
$P_O\mathcal E$ isometrically onto
$P_{gO}\mathcal E$. Moreover,
\[
 \sigma_g(md)=m^g\sigma_g(d)
 \qquad(d\in B_{X,O}).
\]
Thus conjugation by $U_g$ transports
the multiplier defined by $m$ on
$B_{X,O}$ to the multiplier defined
by $m^g$ on $B_{X,gO}$. By uniqueness
of the nondegenerate multiplier
extension,
\[
 U_g\widetilde\pi_O(m)U_g^{-1}
 =\widetilde\pi_{gO}(m^g)
\]
on $P_{gO}\mathcal E$.
It also holds for the zero-extended operators on $\mathcal E$, since $U_gP_OU_g^{-1}=P_{gO}$.

Now fix the source slice $S_g$ of $D_{ij}(C)$. Regard $a$ and $\Theta_C(a)$ as sections on $X\sqcup Y$ by extension by zero. For $j>0$ and $i>0$, respectively, put
\[
 m_j=\Theta_{B_j}^{-1}(a|_{V_j}),
 \qquad
 n_i=\Theta_{B_i}^{-1}\bigl(\Theta_C(a)|_{V_i}\bigr).
\]
For the zeroth coordinate, set $O_0=X$, $m_0=a|_X$, and $n_0=\Theta_C(a)|_X$. Thus $m_j$ is defined on $O_j$ and $n_i$ on $O_i$. For $x\in S_g$, the groupoid-action identity gives
\[
 n_i(gx)=\Theta_{(g,x)}(m_j(x)).
\]
Equivalently, $n_i=m_j^g$ on the Borel set $gS_g\subseteq O_i\cap gO_j$. Lemma~\ref{lem:borel-localization} therefore gives
\[
 P_{gS_g}\widetilde\pi_{O_i}(n_i)
 =P_{gS_g}\widetilde\pi_{gO_j}(m_j^g).
\]
Using local multiplier covariance and $U_gP_{S_g}=P_{gS_g}U_g$, we obtain
\[
 U_gP_{S_g}\widetilde\pi_{O_j}(m_j)
 =\widetilde\pi_{O_i}(n_i)U_gP_{S_g}.
\]
Summing over the finitely many source slices and restricting to $\mathcal E_j$ and $\mathcal E_i$ gives on the $(i,j)$ block
\[
 W_{ij}(C)\,\widehat\pi_j(a)
 =\widehat\pi_i(\Theta_C(a))\,W_{ij}(C).
\]
Summing the finitely many blocks proves \eqref{eq:covariance}. The restriction formula \eqref{eq:restriction} follows directly from the entries: replacing $C$ by $C'$ restricts every $D_{ij}(C)$ to the corresponding source and range bands.
\end{proof}
\begin{rem}
For $p\neq2$, each $\widehat V_C$ is a spatial partial isometry:
it has contractive generalized inverse $\widehat V_{C^{-1}}$, and
their products are the spatial idempotents $Q_{s(C)}$ and
$Q_{r(C)}$ (multiplication-band projections); see
\cite[Proposition~6.5]{GardellaSurvey}. For $p=2$, it is a
Hilbert-space partial isometry with
$\widehat V_{C^{-1}}=\widehat V_C^*$, but need not be spatial
relative to the chosen $L^2$ realization, since its support
projections may be spectral rather than multiplication projections.
\end{rem}

\subsection{The integrated linking representation}

For $f\in\Gamma_c(\cL,r^*\cB)$ supported in a relatively compact open bisection $C$, define
\[
 a_f^C\in\Gamma_c(r(C),\cB),
 \qquad
 a_f^C(r(\gamma))=f(\gamma).
\]
Its support is compactly contained in $r(C)$, so we regard it, by extension by zero, as an element of $\BX\oplus\BY$. Set
\begin{equation}\label{eq:integrated}
 \widehat\Psi(f)=\widehat\pi(a_f^C)\widehat V_C.
\end{equation}
This is independent of the chosen bisection containing $\supp(f)$. Indeed, suppose $C'\subseteq C$ and $\supp(f)\subseteq C'$. For $x\in r(C')$, the unique arrow of $C$ with range $x$ is already the unique arrow of $C'$, so $a_f^C(x)=a_f^{C'}(x)$. If $x\in r(C)\setminus r(C')$, the unique $\gamma\in C$ with $r(\gamma)=x$ lies outside $\supp(f)$, and hence $a_f^C(x)=f(\gamma)=0$. Thus $a_f^C$ is exactly the extension by zero of $a_f^{C'}$, and
\[
 \widehat\pi(a_f^C)=\widehat\pi(a_f^C)Q_{r(C')}.
\]
Together with \eqref{eq:restriction}, this gives
\[
 \widehat\pi(a_f^C)\widehat V_C
 =\widehat\pi(a_f^C)Q_{r(C')}\widehat V_C
 =\widehat\pi(a_f^{C'})\widehat V_{C'}.
\]
Passing to common bisection refinements therefore gives a well-defined linear map on $\Gamma_c(\cL,r^*\cB)$.

It is multiplicative. If $f$ and $k$ are supported in bisections $C$ and $D$, respectively, and $\gamma=cd\in CD$ with $c\in C$ and $d\in D$, then
\[
 (f*k)(cd)=f(c)\,\Theta_c(k(d)).
\]
Equivalently, if $x=r(c)$,
\[
 a_{f*k}^{CD}(x)
 =a_f^C(x)\,\Theta_c\bigl(a_k^D(s(c))\bigr).
\]
To apply covariance, localize the second coefficient to $s(C)$. Choose $\chi\in C_c(s(C))$ with $\chi=1$ on the compact set $s(\supp f)$, and put
\[
 m=\chi\,a_k^D|_{s(C)}\in\Gamma_0(s(C),\cB).
\]
Since $a_f^C$ is supported in $r(\supp f)$,
\[
 \widehat\pi(a_f^C)=\widehat\pi(a_f^C)Q_{r(\supp f)}.
\]
At the level of the $D_{ij}(C)$ entries, restricting to the Borel sub-bisection $\supp(f)\subseteq C$ amounts to cutting down by its source or range band, exactly as in the proof of \eqref{eq:restriction}. In particular,
\[
 Q_{r(\supp f)}\widehat V_C
 =\widehat V_CQ_{s(\supp f)}.
\]
No coefficient section is transported along this Borel sub-bisection; the coefficient transport continues to take place along the open bisection $C$.
Moreover, $m$ agrees with $a_k^D$ on $s(\supp f)$. Because the band projections commute with coefficient multiplication, these identities justify replacing $a_k^D$ by $m$ inside the product. Hence
\begin{align*}
 \widehat\Psi(f)\widehat\Psi(k)
 &=\widehat\pi(a_f^C)\widehat V_C
   \widehat\pi(m)\widehat V_D\\
 &=\widehat\pi(a_f^C)\widehat\pi(\Theta_C(m))
   \widehat V_C\widehat V_D\\
 &=\widehat\pi\bigl(a_f^C\,\Theta_C(m)\bigr)
   \widehat V_{CD}\\
 &=\widehat\pi(a_{f*k}^{CD})\widehat V_{CD}
 =\widehat\Psi(f*k).
\end{align*}
Here the second equality is covariance, the third uses multiplicativity of $\widehat\pi$ and \eqref{eq:Vmult}, and the fourth is the preceding pointwise convolution formula.

Thus
\[
 \widehat\Psi:\Gamma_c(\cL,r^*\cB)\longrightarrow B(\widehat{\mathcal E})
\]
is an algebra homomorphism. For a bisection-supported section,
\begin{equation}\label{eq:bisection-bound}
 \|\widehat\Psi(f)\|\leq\|f\|_\infty.
\end{equation}

\begin{prop}[Exact corner recovery]
Let $f_G\in C_c(G,\BXc)$, viewed as a section of $\Gamma_c(\cG,r^*\cBX)$ through \eqref{eq:groupoid-crossed-identification}. Then
\[
 \widehat\Psi(f_G)=(\pi\rtimes U)(f_G)\oplus0
\]
with respect to $\widehat{\mathcal E}=\mathcal E_0\oplus_p\bigoplus_{i>0}\mathcal E_i$.
\end{prop}

\begin{proof}
It suffices to consider $f_G$ supported in a relatively compact open bisection $C\subseteq\cG$. Every arrow of $C$ has source and range in $X=R_0$, so
\[
 D_{00}(C)=C,
 \qquad
 D_{ij}(C)=\varnothing\quad((i,j)\neq(0,0)).
\]
Write
\[
 C=\bigsqcup_{g\in F}\{g\}\times S_g,
\]
where $F\subseteq G$ is finite and the $S_g$ are open subsets of $X$. Then
\[
 W_C=\sum_{g\in F}U_gP_{S_g},
 \qquad
 \widehat V_C=W_C\oplus0.
\]
Moreover, since $a_{f_G}^C$ is supported over $X$, its representation $\widehat\pi(a_{f_G}^C)$ acts as $\pi(a_{f_G}^C)$ on the zeroth block and vanishes on the other blocks.

For each $g\in F$, the identification \eqref{eq:groupoid-crossed-identification}
gives
\[
 a_{f_G}^C|_{gS_g}
   =f_G(g)|_{gS_g}.
\]
Furthermore, $f_G(g)$ vanishes outside $gS_g$, and Borel covariance gives
\[
 U_gP_{S_g}=P_{gS_g}U_g.
\]
Consequently,
\[
\begin{aligned}
 \pi(a_{f_G}^C)W_C
 &=\sum_{g\in F}
   \pi(a_{f_G}^C)U_gP_{S_g}\\
 &=\sum_{g\in F}
   \pi(f_G(g))U_gP_{S_g}\\
 &=\sum_{g\in F}
   \pi(f_G(g))U_g\\
 &=(\pi\rtimes U)(f_G).
\end{aligned}
\]
Equation~\eqref{eq:integrated} therefore yields
\[
 \widehat\Psi(f_G)
 =(\pi\rtimes U)(f_G)\oplus0.
\]
The general case follows by a finite bisection decomposition.
\end{proof}

The same construction, with $X$ and $Y$ interchanged, dilates every nondegenerate contractive covariant representation of $(H,\BY,\tau)$.

\section{The linking norm and the diagonal corners}

Let $\mathcal R_X$ denote the family of homomorphisms $\widehat\Psi$ obtained by dilating all nondegenerate contractive covariant $L^p$-representations of $(G,\BX,\sigma)$, and let $\mathcal R_Y$ be the symmetric family obtained from $(H,\BY,\tau)$. For $f\in\Gamma_c(\cL,r^*\cB)$ define
\begin{equation}\label{eq:link-norm}
 \|f\|_{\cL,p}
 =\max\left\{
 \sup_{\Psi\in\mathcal R_X}\|\Psi(f)\|,
 \sup_{\Phi\in\mathcal R_Y}\|\Phi(f)\|
 \right\}.
\end{equation}
This is a submultiplicative seminorm. Let $\mathfrak L^p(\cL,\cB)$ be the Hausdorff completion after quotienting by its null space.

If $K\subseteq\cL$ is compact, choose finitely many relatively compact open bisections $C_1,\ldots,C_N$ covering $K$ and a subordinate partition of unity. The estimate \eqref{eq:bisection-bound} gives
\begin{equation}\label{eq:IL-cont}
 \|f\|_{\cL,p}\leq N\|f\|_\infty
 \qquad(\supp(f)\subseteq K).
\end{equation}
Thus uniform convergence on a fixed compact support implies convergence in the linking norm.

\begin{lem}\label{lem:opposite-corner}
Let $\Psi\in\mathcal R_X\cup\mathcal R_Y$. The restriction of $\Psi$ to either diagonal corner, compressed to the corresponding range $Q_X\widehat{\mathcal E}$ or $Q_Y\widehat{\mathcal E}$, is the integrated form of a nondegenerate contractive covariant $L^p$-representation of the corresponding dynamical system.
\end{lem}

\begin{proof}
We treat the $G$-corner. Let
\[
 \Psi_X(f)=Q_X\Psi(f)|_{Q_X\widehat{\mathcal E}}
 \qquad(f\in C_c(G,\BXc))
\]
denote the compressed corner representation. By Lemma~\ref{lem:pihat}, or its symmetric version when $\Psi\in\mathcal R_Y$, the induced representation of $\BX$ on $Q_X\widehat{\mathcal E}$ is nondegenerate. If
\[
 f_G=\sum_{g\in F}d_g\delta_g\in C_c(G,\BXc),
\]
then \eqref{eq:bisection-bound} and a decomposition into the finitely many group slices give
\[
 \|\Psi_X(f_G)\|\leq\sum_{g\in F}\|d_g\|=\|f_G\|_1.
\]
Hence $\Psi_X$ extends contractively to the Banach algebra $\ell^1(G,\BX,\sigma)$ and remains nondegenerate.

The contractive approximate identity of $\BX$ gives one for $\ell^1(G,\BX,\sigma)$, so the corner representation extends uniquely and contractively to its multiplier algebra. Let
\[
 i_X:\BX\longrightarrow M(\ell^1(G,\BX,\sigma))
\]
be the canonical coefficient map, and for $g\in G$ let $v_g$ be the invertible multiplier
\[
 (v_gf)(t)=\sigma_g(f(g^{-1}t)),
 \qquad
 (fv_g)(t)=f(tg^{-1}).
\]
Both $v_g$ and $v_{g^{-1}}=v_g^{-1}$ have norm one. Put
\[
 \pi_\Psi(d)=\overline{\Psi_X}(i_X(d)),
 \qquad
 U_g^\Psi=\overline{\Psi_X}(v_g).
\]
Then $U_g^\Psi U_{g^{-1}}^\Psi=U_{g^{-1}}^\Psi U_g^\Psi=I$, while contractivity gives
\[
 \|U_g^\Psi\|\leq1,
 \qquad
 \|U_{g^{-1}}^\Psi\|\leq1.
\]
Thus $U_g^\Psi$ and its inverse are contractions, so $U_g^\Psi$ is a surjective isometry. The multiplier relations also give
\[
 U_g^\Psi\pi_\Psi(d)(U_g^\Psi)^{-1}
 =\pi_\Psi(\sigma_g(d)).
\]
Moreover $\pi_\Psi$ is nondegenerate. Indeed, for any contractive approximate identity $(u_\lambda)$ of $\BX$,
\[
 \pi_\Psi(u_\lambda)=\Psi_X(u_\lambda\delta_{e_G})
 \longrightarrow I
\]
strongly. Hence $(\pi_\Psi,U^\Psi)$ is a nondegenerate contractive covariant representation. For $d\in\BX$ and $g\in G$,
\[
 \pi_\Psi(d)U_g^\Psi
 =\overline{\Psi_X}(i_X(d))\,\overline{\Psi_X}(v_g)
 =\overline{\Psi_X}(i_X(d)v_g)
 =\Psi_X(d\delta_g),
\]
because $i_X(d)v_g=d\delta_g$ in the canonical multiplier realization of $\ell^1(G,\BX,\sigma)$. By linearity the integrated form $\pi_\Psi\rtimes U^\Psi$ agrees with $\Psi_X$ on $C_c(G,\BXc)$ and therefore on $C_c(G,\BX)$ by density. The $H$-corner is symmetric.
\end{proof}

\begin{prop}[Diagonal norm equality]\label{prop:diagonal-norm}
For $f_G\in C_c(G,\BXc)$ and $f_H\in C_c(H,\BYc)$,
\[
 \|f_G\|_{\cL,p}=\|f_G\|_{F^p(G,\BX,\sigma)},
 \qquad
 \|f_H\|_{\cL,p}=\|f_H\|_{F^p(H,\BY,\tau)}.
\]
Consequently the closures of the two diagonal corners in $\mathfrak L^p(\cL,\cB)$ are isometrically $F^p(G,\BX,\sigma)$ and $F^p(H,\BY,\tau)$.
\end{prop}

\begin{proof}
Exact corner recovery gives
\[
 \|f_G\|_{\cL,p}\geq\|f_G\|_{F^p(G,\BX,\sigma)}.
\]
Conversely, Lemma~\ref{lem:opposite-corner} shows that every representation occurring in \eqref{eq:link-norm}, when restricted to the $G$-corner, is an admissible covariant representation for the intrinsic full crossed product norm. Hence the reverse inequality holds. The $H$-corner is symmetric, and the assertion about the completions follows from the density established after Lemma~\ref{lem:induced-basic}.
\end{proof}

\section{Fullness and symmetric imprimitivity}

We now prove fullness and nondegeneracy of the off-diagonal corners. The argument is a coefficient version of the pairing-valued approximate-identity construction in \cite{CDW}.

\begin{lem}\label{lem:pairing-cai}
There is a contractive approximate identity $(u_\lambda)$ for $\BX$, with $u_\lambda\in\BXc$, such that each $u_\lambda\delta_{e_G}\in C_c(G,\BXc)$ is a finite sum
\begin{equation}\label{eq:pairing-cai}
 u_\lambda\delta_{e_G}
 =\sum_{k=1}^{n_\lambda}\xi_{\lambda,k}*\eta_{\lambda,k},
\end{equation}
where
\[
 \xi_{\lambda,k}\in\Gamma_c(Z,r^*\cB),
 \qquad
 \eta_{\lambda,k}\in\Gamma_c(Z^{\op},r^*\cB).
\]
A symmetric statement holds for $\BY$.
\end{lem}

\begin{proof}
Use the approximate identity constructed in Lemma~\ref{lem:induced-basic}. The images under $q_H$ of simultaneous slices form a basis for $X$: given $zH\in O\subseteq X$, apply Lemma~\ref{lem:simultaneous-slices} inside the open set $q_H^{-1}(O)$. Hence the finitely many trivializing neighborhoods used in each stage of that construction may be chosen as
\[
 U_j=r(S_j)=q_H(S_j),
\]
where $S_j\subseteq Z$ is a simultaneous slice, hence an open bisection of $\cL$. The corresponding approximate-identity element has the form
\[
 u=\sum_{j=1}^m\varphi_j a_j^{(j)},
\]
with $\varphi_j\in C_c(U_j)$ nonnegative and $a_j\in A$.

Because $A$ has a bounded approximate identity, it is a nondegenerate left $A$-module. Cohen--Hewitt factorization therefore gives $a_j=b_jc_j$ for some $b_j,c_j\in A$. Define local sections
\[
 r_j=\sqrt{\varphi_j}\,b_j^{(j)},
 \qquad
 s_j=\sqrt{\varphi_j}\,c_j^{(j)}.
\]
Their supports are contained in the compact set $\supp(\varphi_j)\subseteq U_j$, so extension by zero gives $r_j,s_j\in\BXc$, and
\[
 \varphi_j a_j^{(j)}=r_js_j.
\]

Define $\xi_j$ on $S_j\subseteq Z$ by
\[
 \xi_j(z)=r_j(r(z)).
\]
For $z\in S_j$, the value $s_j(r(z))$ lies in $(\cBX)_{r(z)}$ and $\Theta_z:(\cBY)_{s(z)}\to(\cBX)_{r(z)}$ is an isometric algebra isomorphism. Hence we may define, unambiguously,
\[
 \eta_j(z^{\op})=\Theta_{z^{\op}}(s_j(r(z)))
 =\Theta_z^{-1}(s_j(r(z))).
\]

Both sections are continuous because
the bundle action is continuous on
the open bisection $S_j$, and they
have compact support because
$r|_{S_j}$ is a homeomorphism and
$r_j,s_j$ are compactly supported
in $U_j$. Extend $\xi_j$ and
$\eta_j$ by zero to $Z$ and
$Z^{\op}$, respectively.

Since $S_j$ is a bisection,
\[
 \supp(\xi_j*\eta_j)
 \subseteq S_jS_j^{-1}
 =r(S_j)\subseteq X.
\]
For $x\in U_j$,
there is a unique $z\in S_j$ with
$r(z)=x$, and \eqref{eq:conv} gives
\[
 (\xi_j*\eta_j)(x)=r_j(x)s_j(x).
\]
Consequently,
\[
 \xi_j*\eta_j
 =(\varphi_j a_j^{(j)})\delta_{e_G}.
\]
Summing over $j$ proves
\eqref{eq:pairing-cai}. The construction
for $\BY$ is symmetric, using
bisections in $Z^{\op}$.
\end{proof}

Let
\begin{align*}
 X_p&=\overline{\Gamma_c(Z^{\op},r^*\cB)}^{\,\|\cdot\|_{\cL,p}},\\
 Y_p&=\overline{\Gamma_c(Z,r^*\cB)}^{\,\|\cdot\|_{\cL,p}}.
\end{align*}
At the algebraic level,
\[
 Z^{\op}Z\subseteq\cH,
 \qquad
 ZZ^{\op}\subseteq\cG.
\]
Thus convolution sends
\[
 \Gamma_c(Z^{\op},r^*\cB)\times\Gamma_c(Z,r^*\cB)
 \longrightarrow \Gamma_c(\cH,r^*\cBY)
\]
and, in the reverse order, into $\Gamma_c(\cG,r^*\cBX)$. Proposition~\ref{prop:diagonal-norm} identifies the closures of these diagonal algebras with the corresponding crossed products. Since multiplication in $\mathfrak L^p(\cL,\cB)$ is continuous, it therefore gives bimodule actions and pairings
\[
 \langle\cdot,\cdot\rangle_{F^p(H,\BY,\tau)}:
 X_p\times Y_p\longrightarrow F^p(H,\BY,\tau),
\]
and
\[
 {}_{F^p(G,\BX,\sigma)}\!\langle\cdot,\cdot\rangle:
 Y_p\times X_p\longrightarrow F^p(G,\BX,\sigma).
\]
Submultiplicativity gives the pairing and module norm estimates, while associativity of convolution gives the correspondence identities and the two Morita compatibility identities. Thus $(X_p,Y_p)$ is a Banach correspondence from $F^p(G,\BX,\sigma)$ to $F^p(H,\BY,\tau)$ in the convention above, and $(Y_p,X_p)$ is the opposite correspondence.

\begin{prop}\label{prop:full-nondegenerate}
These two Banach correspondences are full and nondegenerate.
\end{prop}

\begin{proof}
Let $(u_\lambda)$ be the approximate identity from Lemma~\ref{lem:pairing-cai}. If
\[
 f_G=\sum_{g\in F}d_g\delta_g\in C_c(G,\BXc),
\]
then
\[
 ((u_\lambda\delta_{e_G})*f_G)(g)=u_\lambda d_g,
\]
so
\[
 \|(u_\lambda\delta_{e_G})*f_G-f_G\|_1
 =\sum_{g\in F}\|u_\lambda d_g-d_g\|\longrightarrow0.
\]
Hence the convergence also holds in the full crossed product norm. By \eqref{eq:pairing-cai} and associativity,
\[
 (u_\lambda\delta_{e_G})*f_G
 =\sum_k\xi_{\lambda,k}*(\eta_{\lambda,k}*f_G),
\]
so the linear span of the $G$-valued pairing is dense in $F^p(G,\BX,\sigma)$. The $H$-valued pairing is full by symmetry.

For module nondegeneracy, first let $\xi\in\Gamma_c(Z,r^*\cB)$ be supported in an open bisection $S\subseteq Z$. Since $r|_S:S\to r(S)$ is a homeomorphism, define
\[
 d_\xi(x)=\xi((r|_S)^{-1}(x))\qquad(x\in r(S))
\]
and extend by zero outside $r(S)$. Because $\supp(\xi)$ is compactly contained in $S$, this extension belongs to $\BXc$. For $z\in S$,
\[
 ((u_\lambda\delta_{e_G})*\xi)(z)
 =u_\lambda(r(z))\xi(z),
\]
so
\[
 \sup_{z\in S}
 \|((u_\lambda\delta_{e_G})*\xi)(z)-\xi(z)\|
 =\|u_\lambda d_\xi-d_\xi\|_{\BX}\longrightarrow0.
\]
Moreover the difference is supported in the fixed compact set $\supp(\xi)$. A general compactly supported off-diagonal section is a finite sum of sections supported in open bisections, so the same conclusion holds on a fixed compact support. Estimate \eqref{eq:IL-cont} then gives convergence in the $Y_p$-norm. Thus the left $F^p(G,\BX,\sigma)$-action on $Y_p$ is nondegenerate. The other three module actions are handled in the same way, using the two-sided approximate identities and symmetry.
\end{proof}

\begin{proof}[Proof of the full assertion in Theorem A]
By Proposition~\ref{prop:diagonal-norm},
the diagonal corners are the stated
full crossed products. The off-diagonal
correspondences satisfy the pairing
estimates and Morita compatibility
identities by submultiplicativity
and associativity. They are full
and nondegenerate by
Proposition~\ref{prop:full-nondegenerate},
and hence implement a Morita equivalence.
\end{proof}

\section{Reduced symmetric imprimitivity}

\subsection{Fiber reindexing}

We now restrict the linking construction to regular covariant representations and show that regularity is preserved on the opposite diagonal corner.

Let $\pi_0:\BX\to B(E_0)$ be a nondegenerate contractive representation on a $\sigma$-finite $L^p$-space, and write
\[
 \operatorname{Reg}_G(\pi_0)=\pi_0^{\mathrm{reg}}\rtimes\lambda_G
\]
for the regular representation from \eqref{eq:regular-conventions}. Let
\[
 \rho_{\pi_0}:C_0(X)\longrightarrow B(E_0)
\]
be the central representation obtained from $\pi_0$ as in \eqref{eq:rho}. Using the bounded-Borel extension of the central representation as in Subsection~3.1 (with spectral projections on the original Hilbert space when $p=2$), write $P_S^{(0)}$ for the projection associated with a Borel set $S\subseteq X$. The central representation associated with $\pi_0^{\mathrm{reg}}$ then satisfies
\begin{equation}\label{eq:regular-band}
 (P_S\xi)(g)=P_{g^{-1}S}^{(0)}\xi(g)
 \qquad(g\in G).
\end{equation}
Indeed, if $M_\varphi$ denotes the central multiplier of $\BX$ associated with $\varphi\in C_0(X)$, then
\[
 \sigma_{g^{-1}}(M_\varphi d)
 =M_{\varphi_g}\,\sigma_{g^{-1}}(d),
 \qquad
 \varphi_g(x)=\varphi(gx).
\]
Consequently the central $C_0(X)$-representation associated with $\pi_0^{\mathrm{reg}}$ acts on the $g$-coordinate by
\[
 \xi(g)\longmapsto\rho_{\pi_0}(\varphi_g)\xi(g).
\]
Passing from $\varphi$ to $1_S$ gives $1_S(gx)=1_{g^{-1}S}(x)$ and hence the displayed formula.

Apply the matrix dilation of Section~3 to $\operatorname{Reg}_G(\pi_0)$. We retain the coordinate data
\[
 (B_i,O_i,R_i,T_i,\theta_i)_{i\geq1}
\]
from Subsection~3.2. Throughout Section~6, $\widehat{\mathcal E}$, $Q_W$, $\widehat\pi$, and $\widehat V_C$ denote the dilated space, Borel band projections, coefficient representation, and bisection operators constructed in Subsections~3.2--3.3 from this regular representation. In particular, $\widehat\pi$ is the direct-sum representation of $\BX\oplus\BY$ from Lemma~\ref{lem:pihat}.

For $i>0$, let
\[
 \omega_i:O_i\longrightarrow B_i,
 \qquad
 \omega_i(u)=(s|_{B_i})^{-1}(u),
\]
and put
\[
 z_i(u)=\bigl(\omega_i(u)\bigr)^{-1}\in Z.
\]
For $\nu=(g,i)\in G\times\mathbb N_{>0}$ set
\[
 \Delta_\nu=g^{-1}T_i,
 \qquad
 O_\nu=g^{-1}O_i,
\]
and define
\[
 z_\nu^{\mathrm{loc}}(x)=g^{-1}z_i(gx)\quad(x\in O_\nu),
 \qquad
 z_\nu=z_\nu^{\mathrm{loc}}|_{\Delta_\nu}.
\]
Here the superscript $\mathrm{loc}$ indicates the continuous local coordinate map on the open set $O_\nu$; the undecorated $z_\nu$ denotes its restriction to the generally only Borel set $\Delta_\nu$. The notation $g^{-1}z$ refers to the given left action of $G$ on $Z$.

For each $x\in X$, the fiber $q_H^{-1}(x)$ is a single free right $H$-orbit.

\begin{lem}[Fiber reindexing]\label{lem:fiber-reindexing}
For every $x\in X$, the map
\[
 \Phi_x:\{\nu:x\in\Delta_\nu\}\longrightarrow q_H^{-1}(x),
 \qquad
 \Phi_x(\nu)=z_\nu(x),
\]
is a bijection. There are Borel sets $S_{\nu,h}\subseteq\Delta_\nu$, indexed by $\nu=(g,i)$ with $i>0$ and $h\in H$, such that
\begin{enumerate}
 \item for each fixed $h$, the sets $(S_{\nu,h})_\nu$ form a Borel partition of $X$;
 \item for each fixed $\nu$, the sets $(S_{\nu,h})_h$ form a Borel partition of $\Delta_\nu$;
 \item the map
 \[
 J:\bigoplus_\nu^pP_{\Delta_\nu}^{(0)}E_0\longrightarrow\ell^p(H,E_0),
 \qquad
 (J\xi)(h)=\sum_\nu P_{S_{\nu,h}}^{(0)}\xi_\nu,
 \]
 is an invertible isometry, with
 \[
 (J^{-1}\eta)_\nu=\sum_hP_{S_{\nu,h}}^{(0)}\eta(h).
 \]
\end{enumerate}
\end{lem}

\begin{proof}
Fix $x\in X$ and $z\in q_H^{-1}(x)$, and put $y=q_G(z)\in Y$. There is a unique $i>0$ with $y\in R_i$. Since $\theta_i:T_i\to R_i$ is a bijection, there is a unique $t\in T_i$ with $\theta_i(t)=y$, equivalently $q_G(z_i(t))=y$. Thus $z_i(t)$ and $z$ lie in the same left $G$-orbit. Freeness gives a unique $g\in G$ such that
\[
 z_i(t)=gz.
\]
Taking ranges gives $t=gx$, and hence
\[
 z=g^{-1}z_i(gx)=\Phi_x(g,i).
\]
The same argument gives uniqueness, proving that $\Phi_x$ is bijective.

By Lemma~\ref{lem:simultaneous-slices}, the quotient map $q_H:Z\to X$ is a surjective local homeomorphism. Since $X$ is second countable, choose a countable open cover $(\mathcal U_n)$ admitting continuous local sections $\zeta_n:\mathcal U_n\to Z$. Define
\[
 \Xi_1=\mathcal U_1,
 \qquad
 \Xi_n=\mathcal U_n\setminus\bigcup_{m<n}\mathcal U_m\quad(n\geq2).
\]
Then $(\Xi_n)$ is a Borel partition of $X$ with $\Xi_n\subseteq\mathcal U_n$. Define $\zeta(x)=\zeta_n(x)$ for $x\in\Xi_n$. This countable piecewise definition gives a Borel section
\[
 \zeta:X\longrightarrow Z,
 \qquad
 q_H(\zeta(x))=x.
\]
For $x\in\Delta_\nu$, the points $z_\nu(x)$ and $\zeta(x)$ lie in the same fiber, which is a single free right $H$-orbit. Hence there is a unique $\kappa_\nu(x)\in H$ such that
\[
 z_\nu(x)=\zeta(x)\kappa_\nu(x)^{-1}.
\]
Thus the fiberwise reindexing is summarized by
\[
\begin{tikzcd}[column sep=huge,row sep=small]
 \{\nu:x\in\Delta_\nu\}
   \arrow[r,"\Phi_x","\cong"']
 & q_H^{-1}(x)
   \arrow[r,"z=\zeta(x)h^{-1}\leftrightarrow h","\cong"']
 & H
 \\[-0.4ex]
 \nu \arrow[r,mapsto]
 & z_\nu(x) \arrow[r,mapsto]
 & \kappa_\nu(x).
\end{tikzcd}
\]
The map $\kappa_\nu$ is Borel: for fixed $h\in H$, the set where $\kappa_\nu=h$ is the equality set of the Borel maps $x\mapsto z_\nu(x)$ and $x\mapsto\zeta(x)h^{-1}$ into the Hausdorff space $Z$. Put
\[
 S_{\nu,h}=\{x\in\Delta_\nu:\kappa_\nu(x)=h\}.
\]
Bijectivity of $\Phi_x$ gives the two partition statements.

The projections $P_{S_{\nu,h}}^{(0)}$ are pairwise disjoint in the two ways just described. Applying \eqref{eq:band-pythagoras} in each coordinate gives
\[
 \|J\xi\|_p^p
 =\sum_h\sum_\nu\|P_{S_{\nu,h}}^{(0)}\xi_\nu\|^p
 =\sum_\nu\|P_{\Delta_\nu}^{(0)}\xi_\nu\|^p.
\]
The same calculation verifies the displayed inverse formula.
\end{proof}

By \eqref{eq:regular-band},
\[
 P_{T_i}\ell^p(G,E_0)
 =\{\xi:\xi(g)\in P_{g^{-1}T_i}^{(0)}E_0\text{ for every }g\in G\}
 \cong\bigoplus_{g\in G}^pP_{g^{-1}T_i}^{(0)}E_0.
\]
Therefore the $Y$-corner of the dilated space is canonically
\begin{equation}\label{eq:Y-corner-regular}
 Q_Y\widehat{\mathcal E}
 =\bigoplus_{i>0}^pP_{T_i}\ell^p(G,E_0)
 \cong\bigoplus_{\nu=(g,i),\,i>0}^pP_{\Delta_\nu}^{(0)}E_0,
\end{equation}
where the last identification merely reindexes the countable double $\ell^p$-sum by $\nu=(g,i)$. We use $J$ to identify this space with $\ell^p(H,E_0)$. Whenever an operator $T$ on $\widehat{\mathcal E}$ preserves $Q_Y\widehat{\mathcal E}$, the notation $JTJ^{-1}$ means $J(T|_{Q_Y\widehat{\mathcal E}})J^{-1}$.

\subsection{Preservation of regularity}

For each open $O\subseteq X$, let
\[
 \widetilde\pi_{0,O}:M(B_{X,O})\longrightarrow B(P_O^{(0)}E_0)
\]
denote the local multiplier extension associated with the restriction of $\pi_0$ to $B_{X,O}$, as in Lemma~\ref{lem:local-multiplier}. We use the zero-extension convention from Subsection~3.1, so these local multiplier operators act on all of $E_0$. Their compatibility under Borel cutdowns is given by Lemma~\ref{lem:borel-localization}.

For $\nu=(g,i)$ with $i>0$ and $b\in\BY$, define a bounded continuous section over $O_\nu$ by
\[
 m_\nu(b)(x)
 =\Theta_{z_\nu^{\mathrm{loc}}(x)}\bigl(b(q_G(z_\nu^{\mathrm{loc}}(x)))\bigr).
\]
Recall from Subsection~3.2 that
\[
 m_i(b)(u)=\Theta_{z_i(u)}\bigl(b(q_G(z_i(u)))\bigr),
 \qquad u\in O_i.
\]
If $u=gx$ and $y=q_G(z_i(u))$, then
\[
 z_\nu^{\mathrm{loc}}(x)=g^{-1}z_i(u)=(g^{-1},u)z_i(u),
 \qquad
 q_G(z_\nu^{\mathrm{loc}}(x))=q_G(z_i(u))=y,
\]
the second identity following because $q_G$ is constant on left $G$-orbits. Figure~\ref{fig:local-transport} displays the corresponding composable arrows and the induced coefficient-fiber maps.

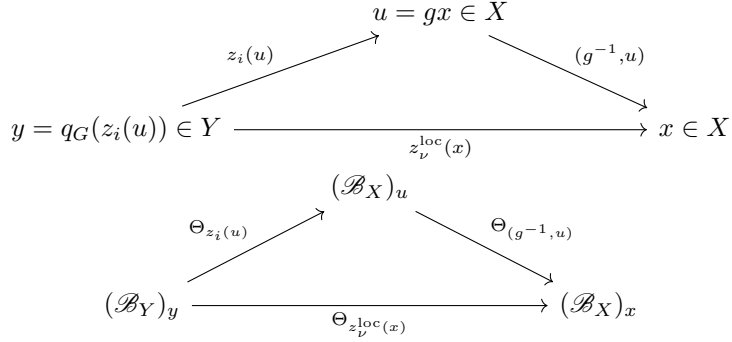
\begin{figure}[htbp]
\centering
\small
\[
\begin{tikzcd}[column sep=huge,row sep=large]
 & u=gx\in X
   \arrow[dr,"{(g^{-1},u)}"]
 & \\
 y=q_G(z_i(u))\in Y
   \arrow[ur,"{z_i(u)}"]
   \arrow[rr,"{z_\nu^{\mathrm{loc}}(x)}"']
 &&
 x\in X
\end{tikzcd}
\]
\vspace{-0.6em}
\[
\begin{tikzcd}[column sep=huge,row sep=large]
 & (\cBX)_u
   \arrow[dr,"{\Theta_{(g^{-1},u)}}"]
 & \\
 (\cBY)_y
   \arrow[ur,"{\Theta_{z_i(u)}}"]
   \arrow[rr,"{\Theta_{z_\nu^{\mathrm{loc}}(x)}}"']
 &&
 (\cBX)_x
\end{tikzcd}
\]
\caption{The same local transport in two levels: groupoid arrows (upper triangle) and induced coefficient-fiber isomorphisms (lower triangle). Both triangles express $z_\nu^{\mathrm{loc}}(x)=(g^{-1},u)z_i(u)$ and $\Theta_{\gamma\eta}=\Theta_\gamma\Theta_\eta$. Unlike the Borel reindexing, these transports use continuous local charts on $O_\nu$.}
\label{fig:local-transport}
\end{figure}

Hence
\begin{align*}
 m_\nu(b)(x)
 &=\Theta_{(g^{-1},u)z_i(u)}
   \bigl(b(q_G(z_i(u)))\bigr)\\
 &=\Theta_{(g^{-1},u)}\!\left(
   \Theta_{z_i(u)}\bigl(b(q_G(z_i(u)))\bigr)
  \right)\\
 &=\Theta_{(g^{-1},u)}\bigl(m_i(b)(u)\bigr)\\
 &=\sigma_{g^{-1}}(m_i(b))(x).
\end{align*}
The second equality uses $\Theta_{\gamma\eta}=\Theta_\gamma\Theta_\eta$, the third is the recalled definition of $m_i(b)$, and the last is the local form of the induced $G$-action on sections of $\cBX$.

Thus the action of $\widehat\pi(0\oplus b)$ on the block $P_{\Delta_\nu}^{(0)}E_0$ in \eqref{eq:Y-corner-regular} is
\[
 P_{\Delta_\nu}^{(0)}\widetilde\pi_{0,O_\nu}(m_\nu(b)).
\]

Put $S_\nu=S_{\nu,e_H}$ and define
\[
 \varrho_0(b)
 =\sum_\nu P_{S_\nu}^{(0)}\widetilde\pi_{0,O_\nu}(m_\nu(b)),
 \qquad b\in\BY.
\]
The sets $S_\nu$ partition $X$, so the projections $P_{S_\nu}^{(0)}$ are pairwise disjoint and have strong sum $I$. Hence the displayed formula defines $\varrho_0(b)$ blockwise on
\[
 E_0=\bigoplus_\nu^pP_{S_\nu}^{(0)}E_0.
\]
Each local multiplier commutes with these band projections and has norm at most $\|b\|$; therefore $\|\varrho_0(b)\|\leq\|b\|$. The local transport maps are algebra isomorphisms, so
\[
 m_\nu(b_1b_2)=m_\nu(b_1)m_\nu(b_2)
 \qquad\text{on }O_\nu.
\]
Consequently $\varrho_0$ is a contractive algebra representation. Its nondegeneracy will follow from the regular-conjugacy calculation below.

\begin{prop}[Regularity is preserved by the linking dilation]\label{prop:regularity-preserved}
Apply the construction of Section~3 with input covariant representation
\[
 (\pi,U)=(\pi_0^{\mathrm{reg}},\lambda_G),
\]
and denote the resulting homomorphism by $\widehat\Psi$. Then $\varrho_0$ is nondegenerate, and under the isometry $J$ from Lemma~\ref{lem:fiber-reindexing}, the compression of $\widehat\Psi$ to the $Y$-corner is $\operatorname{Reg}_H(\varrho_0)$. In particular,
\[
 \|\widehat\Psi(f_H)|_{Q_Y\widehat{\mathcal E}}\|
 \leq\|f_H\|_{F^p_{\red}(H,\BY,\tau)}
 \qquad(f_H\in C_c(H,\BYc)).
\]
\end{prop}

\begin{proof}
We first identify the coefficient representation. If $x\in S_{\nu,h}$, then the defining relation for $S_{\nu,h}$ gives
\[
 z_\nu(x)=\zeta(x)h^{-1}.
\]
For $b\in\BY$, the quotient relation defining $\cBX$ and the definition of $\tau$ in \eqref{eq:induced-actions} give
\[
 \Theta_{\zeta(x)h^{-1}}
 \bigl(b(G\zeta(x)h^{-1})\bigr)
 =\Theta_{\zeta(x)}
 \bigl((\tau_{h^{-1}}b)(G\zeta(x))\bigr).
\]
If $\nu_0$ is the unique index with $x\in S_{\nu_0}$, then $z_{\nu_0}(x)=\zeta(x)$ by definition of $S_{\nu_0}=S_{\nu_0,e_H}$. Therefore, from the definition of $m_{\nu_0}$,
\[
 m_{\nu_0}(\tau_{h^{-1}}b)(x)
 =\Theta_{\zeta(x)}\bigl((\tau_{h^{-1}}b)(G\zeta(x))\bigr),
\]
which is exactly the right-hand side above. Hence, on $S_{\nu,h}\cap S_{\nu_0}$,
\[
 m_\nu(b)=m_{\nu_0}(\tau_{h^{-1}}b).
\]
Applying Lemma~\ref{lem:borel-localization} and summing over the partition $(S_{\nu_0})_{\nu_0}$ yields
\begin{equation}\label{eq:coefficient-regular-conjugacy}
 (J\widehat\pi(0\oplus b)J^{-1}\eta)(h)
 =\varrho_0(\tau_{h^{-1}}b)\eta(h).
\end{equation}

We next identify the group shift. Fix $k\in H$ and take $b\in\BYc$. Choose a relatively compact open set $U\subseteq Y$ containing $\supp(b)$, and consider the relatively compact open bisection
\[
 C=\{(k,y):k\cdot y\in U\}\subseteq H\ltimes Y.
\]
Under the $H$-corner version of \eqref{eq:groupoid-crossed-identification}, the section corresponding to $b\delta_k$ is
\[
 \widetilde{b\delta_k}(h,y)=0\quad(h\neq k),
 \qquad
 \widetilde{b\delta_k}(k,y)=b(k\cdot y).
\]
Its support is therefore contained in $C$.

Suppose the $(\nu,\nu')$ matrix entry of $\widehat V_C$ is nonzero at $x$, where $\nu=(g,i)$ and $\nu'=(g',j)$. The corresponding element of $D_{ij}(C)$ can then be written directly in the coordinate charts as
\[
 d=\omega_i(gx)^{-1}\gamma\,\omega_j(g'x)\in G\ltimes X,
 \qquad \gamma\in C.
\]
Let $\ell\in G$ be the group component of $d$. The regular $G$-shift sends the $g'$-coordinate to the $g$-coordinate, hence $g=\ell g'$. Right-multiplying the displayed factorization by $z_j(g'x)=\omega_j(g'x)^{-1}$ gives
\[
 \ell z_j(g'x)=z_i(gx)\gamma.
\]
Write $\gamma=(k,y)$. Since $r(\gamma)=q_G(z_i(gx))$ and our convention is $h\cdot Gz=G(zh^{-1})$, one has
\[
 y=k^{-1}\cdot q_G(z_i(gx))=G(z_i(gx)k).
\]
Thus the mixed product in the linking groupoid is
\[
 z_i(gx)\gamma=z_i(gx)k,
\]
and hence
\[
 \ell z_j(g'x)=z_i(gx)k.
\]
Multiplying on the left by $g^{-1}=g'^{-1}\ell^{-1}$ gives
\[
 z_\nu(x)=z_{\nu'}(x)k^{-1}.
\]
If $x\in S_{\nu',h'}$, then $z_{\nu'}(x)=\zeta(x)h'^{-1}$, and therefore
\[
 z_\nu(x)=\zeta(x)h'^{-1}k^{-1}=\zeta(x)(kh')^{-1}.
\]
Thus the resulting $H$-coordinate is $h=kh'$, equivalently $h'=k^{-1}h$. This identifies the action of $\widehat V_C$ on its initial band. For every Borel set $W\subseteq Y$, put
\[
 \mathsf P_W
 =J\bigl(Q_W|_{Q_Y\widehat{\mathcal E}}\bigr)J^{-1}.
\]
These projections admit a useful explicit description. Namely, for $h\in H$, set
\[
 X_{W,h}=\{x\in X:q_G(\zeta(x)h^{-1})\in W\}.
\]
The sets $X_{W,h}$ are Borel, and the definition of $J$, together with $q_G(z_\nu(x))=\theta_i(gx)$ for $\nu=(g,i)$, gives
\begin{equation}\label{eq:regular-W-projections}
 (\mathsf P_W\eta)(h)=P_{X_{W,h}}^{(0)}\eta(h).
\end{equation}
Thus $\mathsf P_W$ is the Borel-coordinate projection that tests whether the unit of $Y$ represented by $\zeta(x)h^{-1}$ lies in $W$. In particular, our convention $k\cdot Gz=G(zk^{-1})$ gives
\[
 \lambda_H(k)\mathsf P_{k^{-1}\!\cdot U}
 =\mathsf P_U\lambda_H(k).
\]
We verify the converse to the preceding index calculation. Fix $x\in X$ and $h'\in H$, and let $\nu'=(g',j)$ be the unique index with $x\in S_{\nu',h'}$. By Lemma~\ref{lem:fiber-reindexing}, there is also a unique $\nu=(g,i)$ such that
\[
 z_\nu(x)=z_{\nu'}(x)k^{-1}.
\]
Put $y'=q_G(z_{\nu'}(x))$ and $y=q_G(z_\nu(x))$. Then $y=k\cdot y'$. If $y\in U$, the arrow $\gamma=(k,y')$ lies in $C$, and
\[
 d=\omega_i(gx)^{-1}\gamma\,\omega_j(g'x)\in D_{ij}(C).
\]
Indeed, the source and range of $\gamma$ match the ranges of $\omega_j(g'x)$ and $\omega_i(gx)$, respectively, and $gx\in T_i$, $g'x\in T_j$. Furthermore,
\[
 z_i(gx)k=(gg'^{-1})z_j(g'x),
\]
so the $G$-component of $d$ is $gg'^{-1}$. The regular $G$-shift therefore carries the $g'$-coordinate to the $g$-coordinate with no coefficient operator acting on the underlying $E_0$-vector. Conversely, if $y\notin U$, the defining range condition of $C$ excludes that entry. Thus every source band permitted by $k^{-1}\!\cdot U$ contributes exactly once, at the resulting $H$-coordinate $h=kh'$.

Since $r(C)=U$ and $s(C)=k^{-1}\!\cdot U$, the initial and final projections of $\widehat V_C$ are $Q_{k^{-1}\!\cdot U}$ and $Q_U$, respectively. Formula \eqref{eq:regular-W-projections} identifies these support projections after conjugation by $J$. Combining both directions of the index calculation, the resulting band-supported shift is exactly the left shift restricted to its indicated initial band, giving
\[
 J\widehat V_CJ^{-1}
 =\mathsf P_U\lambda_H(k)
 =\lambda_H(k)\mathsf P_{k^{-1}\!\cdot U}.
\]
Let
\[
 D_b=J\widehat\pi(0\oplus b)J^{-1}.
\]
Because $\supp(b)\subseteq U$, the block construction gives $\widehat\pi(0\oplus b)=\widehat\pi(0\oplus b)Q_U$. After conjugation by $J$, this becomes $D_b=D_b\mathsf P_U$. Hence
\[
 J\widehat\Psi(b\delta_k)|_{Q_Y\widehat{\mathcal E}}J^{-1}
 =D_b\mathsf P_U\lambda_H(k)
 =D_b\lambda_H(k).
\]
Together with \eqref{eq:coefficient-regular-conjugacy} and the transformation groupoid identification \eqref{eq:groupoid-crossed-identification}, this gives
\begin{equation}\label{eq:regular-conjugacy}
 J\widehat\Psi(b\delta_k)|_{Q_Y\widehat{\mathcal E}}J^{-1}
 =\operatorname{Reg}_H(\varrho_0)(b\delta_k).
\end{equation}
By linearity, the same identity holds on $C_c(H,\BYc)$.

It remains to show that $\varrho_0$ is nondegenerate. Lemma~\ref{lem:pihat} says that the $\BY$-representation $\widehat\pi|_{\BY}$ on the $Y$-corner is nondegenerate. By \eqref{eq:coefficient-regular-conjugacy}, the closure of the span of its range applied to $\ell^p(H,E_0)$ is
\[
 \ell^p\bigl(H,\overline{\varrho_0(\BY)E_0}\bigr),
\]
because every $\tau_h$ is onto and finitely supported vectors are dense for $1\leq p<\infty$. This closure is all of $\ell^p(H,E_0)$, so $\overline{\varrho_0(\BY)E_0}=E_0$. Hence $\varrho_0$ is nondegenerate, and the stated reduced-norm estimate follows from \eqref{eq:regular-conjugacy} and the definition of the reduced norm.
\end{proof}

The same proposition holds with $G,\BX,X$ and $H,\BY,Y$ interchanged.

\subsection{The reduced linking completion}

Let $\mathcal R_X^{\red}$ be the family of linking dilations obtained from all regular covariant representations of $(G,\BX,\sigma)$, and let $\mathcal R_Y^{\red}$ be the symmetric family. Define
\[
 \|f\|_{\cL,\red,p}
 =\max\left\{
  \sup_{\Psi\in\mathcal R_X^{\red}}\|\Psi(f)\|,
  \sup_{\Phi\in\mathcal R_Y^{\red}}\|\Phi(f)\|
 \right\}.
\]
It is a finite submultiplicative seminorm and is bounded above by the full auxiliary linking seminorm $\|\cdot\|_{\cL,p}$ from \eqref{eq:link-norm}. Let $\mathfrak L^p_{\red}(\cL,\cB)$ be the Hausdorff completion after quotienting by the null space.

\begin{prop}[Reduced diagonal norm equality]\label{prop:reduced-diagonal-norm}
For $f_G\in C_c(G,\BXc)$ and $f_H\in C_c(H,\BYc)$,
\[
 \|f_G\|_{\cL,\red,p}=\|f_G\|_{F^p_{\red}(G,\BX,\sigma)},
 \qquad
 \|f_H\|_{\cL,\red,p}=\|f_H\|_{F^p_{\red}(H,\BY,\tau)}.
\]
Consequently the closures of the diagonal corners in $\mathfrak L^p_{\red}(\cL,\cB)$ are isometrically the two reduced crossed products.
\end{prop}

\begin{proof}
Exact corner recovery applied to regular covariant representations gives the inequality ``$\geq$'' on the $G$-corner. For the reverse inequality, members of $\mathcal R_X^{\red}$ restrict there to the regular representations from which they were dilated, while Proposition~\ref{prop:regularity-preserved}, applied symmetrically, shows that members of $\mathcal R_Y^{\red}$ also restrict to regular covariant representations. Every term in the supremum defining $\|f_G\|_{\cL,\red,p}$ is therefore bounded by the intrinsic reduced norm. The $H$-corner is symmetric.
\end{proof}

Let
\begin{align*}
 X_{p,\red}
 &=\overline{\Gamma_c(Z^{\op},r^*\cB)}^{\,\|\cdot\|_{\cL,\red,p}},\\
 Y_{p,\red}
 &=\overline{\Gamma_c(Z,r^*\cB)}^{\,\|\cdot\|_{\cL,\red,p}}.
\end{align*}
As in the full case, convolution makes $(X_{p,\red},Y_{p,\red})$ a Banach correspondence from $F^p_{\red}(G,\BX,\sigma)$ to $F^p_{\red}(H,\BY,\tau)$, and the reversed pair is the opposite correspondence. Submultiplicativity gives the norm estimates and associativity gives the correspondence and Morita compatibility identities.

\begin{proof}[Proof of the reduced case in Theorem A]
The pairing-valued approximate identities
from Lemma~\ref{lem:pairing-cai} apply
unchanged. Since the reduced
crossed product norms are bounded by
$\|\cdot\|_1$, the fullness argument in
Proposition~\ref{prop:full-nondegenerate}
also establishes fullness of both
reduced pairings.

Moreover,
\[
 \|\cdot\|_{\cL,\red,p}
 \leq\|\cdot\|_{\cL,p}.
\]
Hence the approximate-identity convergence
on the off-diagonal modules established in
Proposition~\ref{prop:full-nondegenerate}
continues to hold in the reduced linking
norm. All four reduced module actions
are therefore nondegenerate.

Together with
Proposition~\ref{prop:reduced-diagonal-norm}
and the Morita compatibility identities,
this proves the reduced assertion.
\end{proof}

\begin{rem}
When $A=\mathbb C$ and both coefficient actions are trivial, Theorem~A reduces to the full and reduced $L^p$ symmetric imprimitivity theorem in \cite[Theorem~8.1]{CDW}. When $p=2$ and $A$ is a $C^*$-algebra, the theorem gives Banach algebra Morita equivalence for the full crossed products appearing in Raeburn's symmetric imprimitivity theorem \cite{Raeburn88}; our assertion is not formulated in terms of Hilbert $C^*$-imprimitivity bimodules. For reduced crossed products, the corresponding $C^*$-theoretic strong Morita equivalence was
established by Quigg--Spielberg \cite{QuiggSpielberg}; a direct regular representation proof was subsequently given by an Huef--Raeburn \cite{anHuefRaeburn}, while a linking bundle formulation appears in \cite[Section~6]{SimsWilliamsFell}.
\end{rem}

 \section*{Funding}
This work was supported by the Research Start-up Fund of Jilin University (grant number 419080524100).
 \section*{Acknowledgements}
The author used OpenAI's ChatGPT
6-Sol to assist
with mathematical exploration, the
examination of arguments, and manuscript
organization and revision. The author
directed this work, reviewed the resulting
suggestions, and assumes responsibility
for the accuracy and originality of
the manuscript.

\end{document}